\documentclass[12pt]{amsart}
\usepackage{amsmath,amsthm,amsfonts,latexsym,amssymb,amscd,color}
\usepackage{chemarr}
\newtheorem{thm}{Theorem}[subsection]
\newtheorem{cor}[thm]{Corollary}
\newtheorem{lm}[thm]{Lemma}

\newtheorem{clm}[thm]{Claim}
\newtheorem*{clm*}{Claim}
\theoremstyle{definition}
\newtheorem{df}[thm]{Definition}

\newtheorem{exmp}[thm]{Example}
\newtheorem{remark}[thm]{Remark}

\numberwithin{equation}{section}
\newcommand{\sprf}{\noindent{\it Proof.}} %used in spf environment
\newcommand{\sqed}{\hfill\rule{1.3mm}{3mm}\medskip}

\newcommand{\cproof}{\noindent{\it Proof of claim.}\ } %used in spf environment
\newcommand{\cqed}{\hfill\rule{1.3mm}{3mm}}

\DeclareMathOperator{\im}{im}         % imaginary part
\newcommand{\wec}[1]{{\mathbf{#1}}}  % notation for vectors in algebras
\newcommand{\wrel}[1]{\;#1\;}     % a\wrel{\alpha}b = a \alpha b
\newcommand{\m}[1]{{\mathbf{\uppercase{#1}}}}
\newcommand{\vr}[1]{{\mathcal{\uppercase {#1}}}}
\newcommand{\bd}{\begin{description}}
\newcommand{\ed}{\end{description}}
\newcommand{\lb}{\langle} % <...>
\newcommand{\rb}{\rangle}
\newcommand{\Aut}{\textrm{Aut}}
\newcommand{\Inn}{\textrm{Inn}}

\newcommand{\F}{{F}}

\newcommand{\supp}{{\rm supp}}

\begin{document}

\title[Growth Rates of Algebras]{Growth Rates of 
Algebras, I:\\
\vspace{2mm} {\normalsize\rm pointed cube terms}}

\author{Keith A. Kearnes}
\address[Keith Kearnes]{Department of Mathematics\\
University of Colorado\\
Boulder, CO 80309-0395\\
USA}
\email{Keith.Kearnes@Colorado.EDU}
\author{Emil W. Kiss}
\address[Emil W. Kiss]{
Lor\'{a}nd E{\"o}tv{\"o}s University\\
Department of Algebra and Number Theory\\
H--1117 Budapest, P\'{a}zm\'{a}ny P\'{e}ter sétány 1/c.\\
Hungary}
\email{ewkiss@cs.elte.hu}
\author{\'Agnes Szendrei}
\address[\'Agnes Szendrei]{Department of Mathematics\\
University of Colorado\\
Boulder, CO 80309-0395\\
USA}
\email{Agnes.Szendrei@Colorado.EDU}
\thanks{This material is based upon work supported by
the Hungarian National Foundation for Scientific Research (OTKA)
grant no.\ K77409, K83219, and K104251.
}

\begin{abstract}
We investigate the function $d_{\m a}(n)$,
which gives the size of a least size generating
set for $\m a^n$.
\end{abstract}

\maketitle

\section{Introduction}\label{intro_sec}
For a finite algebra $\m a$, write $d_{\m a}(n)=g$
if $g$ is the least size of a generating set
for $\m a^n$, and write $h_{\m a}(g)=n$ if the
largest power of $\m a$ that is $g$-generated is $\m a^n$.
The functions $d_{\m a}$ and 
$h_{\m a}$ map natural numbers
%nonnegative integers
to natural numbers
%nonnegative integers 
and are related by
%\begin{equation}\label{adjoint}
\[
d_{\m a}(n) \leq g 
\wrel{\Longleftrightarrow} 
\textrm{$\m a^n$ is $g$-generated}
\wrel{\Longleftrightarrow} 
n \leq h_{\m a}(g),
\]
%\end{equation}
which asserts that $d_{\m a}$
is the lower adjoint of $h_{\m a}$
and $h_{\m a}$
is the upper adjoint of $d_{\m a}$. 
It follows that 
%$d_{\m a}, h_{\m a}\colon \mathbb Z^{\geq 0}\to\mathbb Z^{\geq 0}$ 
$d_{\m a}, h_{\m a}\colon \omega\to\omega$ 
are increasing functions,
which are inverse bijections between their images:
\[
\im(d_{\m a}) \xrightleftharpoons[d]{h} \im(h_{\m a});
\]
and, moreover, each determines the other.
These functions make sense for partial algebras
and infinite algebras, too.

The study of the functions $d_{\m a}$ and $h_{\m a}$
has a long history, which we briefly survey.

\subsection{The $\phi$-function of a group}
In the 1936 paper \cite{hall}, Philip Hall 
generalizes the Euler $\phi$-function
from number theory by defining $\phi_k(G)$ to be the number of $k$-tuples
$\wec{t}=(t_1,\ldots,t_k)$ for which 
$\{t_1,\ldots,t_k\}$
is a generating set of the group $G$. 
The classical
Euler $\phi$-function is therefore $\phi(k)=\phi_1(\mathbb Z_k)$.
Hall calls two generating $k$-tuples $\wec{t}_1$
and $\wec{t}_2$ ``equivalent''
if there is an automorphism $\alpha$ of $G$ 
which applied coordinatewise to 
$\wec{t}_1$ yields $\wec{t}_2$. 
The automorphism group of $G$
acts freely on generating $k$-tuples, hence the number
of equivalence classes of generating $k$-tuples 
is $\phi_k(G)/|\Aut(G)|$. Hall denotes 
$\phi_k(G)/|\Aut(G)|$ by $d_k(G)$, an unfortunate conflict
with more recent notation since 
$\phi_k(G)/|\Aut(G)|$ is closer to the $h$-function 
than to the $d$-function.
Indeed, if $G$ is a finite simple nonabelian group, then
$h_G(k) = \phi_k(G)/|\Aut(G)|$. 

Hall calls the function $\phi_k(G)/|\Aut(G)|$ 
``intrinsically more interesting'' than $\phi_k(G)$, 
and derives a formula for it in the case where $G$
is a finite simple nonabelian group, namely
\begin{equation}\label{hall_formula}
h_{G}(k) = \frac{1}{|\Aut(G)|}\sum_{H\leq G} \mu(H)|H|^k
\end{equation}
where $\mu$ is the M\"obius function of the subgroup
lattice of $G$. This calculation 
is the first result of our topic.

\subsection{Non-Hopf kernels}
A group is \emph{Hopfian} if every surjective endomorphism is
an isomorphism, and non-Hopfian otherwise. A group $N$
is a \emph{non-Hopf kernel} of $G$ if it is isomorphic to
the kernel of a surjective endomorphism of $G$ that is not
an isomorphism. In the 1969 paper \cite{dey},
I.~M.~S.~Dey investigates
the problem of determining which groups are non-Hopf kernels.
Dey notes that every nontrivial group is a non-Hopf kernel,
since, for example, the kernel of the shift
\[
N^{\omega}\to N^{\omega}\colon (n_0,n_1,n_2,\ldots)\mapsto (n_1,n_2,n_3\ldots)
\]
is isomorphic to $N$.
Dey restricts attention to non-Hopf kernels of finitely
generated groups, and notes the following:
a finite complete group is not a non-Hopf kernel
of a finitely generated group. ($N$ is complete
if it is centerless and $\Aut(N)=\Inn(N)$.)
His reasoning goes like this: if $N$ is complete
and a non-Hopf kernel of $G$, then 
$C_G(N)$ is a normal complement to $N$. 
By the non-Hopf property, $C_G(N)\cong G$, so 
\[
G\cong N\times G\cong N^2\times G\cong N^3\times G\cong 
\cdots.
\]
If $G$ is finitely generated,
say by $g$ elements, then so are the quotient groups
$N^n$ for all finite $n$. But this contradicts the 
local finiteness of the variety ${\mathcal V}(N)$.
Specifically, the $g$-generated groups in this 
variety have size at most $|N|^{|N|^g}$.
Thus, Dey's paper draws attention to the (easy) fact that
if $N$ is finite, then the number of elements
required to generate $N^n$ goes to infinity
as $n$ goes to infinity. (In symbols,
$\lim_{n\to\infty}(d_{N}(n))=\infty$.)

\subsection{Growth rates of groups}
In the 1974 paper \cite{wiegold1},
James Wiegold cites Dey's work on non-Hopf
kernels as the inspiration for his investigation
into the question 
``What are the ways in which \ldots [$d_G(n)$]
\ldots can tend to infinity [when $G$ is a finite group]?''
Wiegold inverts Hall's formula (\ref{hall_formula})
to show that, for $n>0$, $d_{G}(n)$ is one of the three
natural numbers nearest
\[
\log_{|G|}(n) + \log_{|G|}(|\Aut(G)|)
\]
when $G$ is a finite simple nonabelian group,
so in this case $d_G(n)$ is asymptotically equivalent to $\log(n)$.
He shows that $d_G(n)$ has logarithmic upper and lower
bounds whenever $G$ is a finite perfect
group. ($G$ is perfect if $[G,G]=G$.)
He shows also that $d_G(n)$
agrees with a linear function for large $n$ if $G$
is a finite imperfect group.
Thus, he establishes that $d_G(n)$ tends to infinity
as a logarithmic or linear function when $G$ is a finite group.

\subsection{Growth rates of groups, semigroups and group expansions}
Wiegold's paper initiated a program of research
into growth rates of groups including, for example,
\cite{erfanian1,erfanian2,erfanian3,erfanian4,erfanian5,
erfanian_rezaei, erfanian_wiegold, 
kimmerle, lennox_wiegold, meier_wiegold, 
obraztsov, stewart_wiegold, wiegold2, wiegold3, wiegold4, 
wiegold_wilson,wise}.
%\marginpar{When paper \\ is done\\ replace some\\ with range a--b}
The program expanded to include the 
investigation of growth rates of semigroups, in
\cite{pollak, wiegold_sem}, and later to 
include the investigation of
growth rates of more general algebraic structures,
in \cite{glass_riedel, riedel}. 
Some of the questions being investigated about growth rates of
finite algebras are related to the following theorems of Wiegold:
\begin{enumerate}
\item[(I)] 
A finite perfect group
has growth rate 
that is logarithmic ($d_{\m a}(n)\in \Theta(\log(n))$),
while a finite imperfect group has 
growth rate that is linear ($d_{\m a}(n)\in \Theta(n)$).
\item[(II)] 
A finite semigroup with identity
has growth rate that is logarithmic or linear, while a finite
semigroup without identity has
growth rate that is exponential 
($d_{\m a}(n)\in 2^{\Theta(n)}$), \cite{wiegold_sem}.
% Note: Quick and Ruskuc omit the logarithmic case when citing 
% (II) (page 2 of their preprint). 
\end{enumerate}

\noindent
Herbert Riedel partially extends Item (I) to 
congruence uniform varieties in \cite{riedel}
by proving that finite algebras 
in such varieties 
that are perfect
(in the sense of modular commutator theory)
have logarithmic growth rate.
The paper \cite{quick_ruskuc} by 
Martyn Quick and Nikola Ru\v skuc
extends Item (I)
to any variety of rings, modules, $k$-algebras
or Lie algebras, but 
also falls short of extending Item (I)
to arbitrary congruence uniform varieties.

\subsection{Our work}\label{our_work}
We got interested in growth rates of finite
algebras after reading Remark 4.15 of 
\cite{quick_ruskuc}, which 
states that 
\emph{``At present no finite algebraic structure is known
for which the $d$--sequence does not have 
one of logarithmic, linear or exponential growth.''}
We found some of these missing algebras.
(Theorem~\ref{example}.)

Our interest in growth rates
was later strengthened upon learning about
paper \cite{chen}, by Hubie Chen,
which links growth rates with the constraint satisfaction
problem by giving a polynomial time reduction
from the quantified constraint satisfaction problem
to the ordinary constraint satisfaction problem
for algebras with $d_{\m a}(n)\in O(n^k)$ for some $k$.
Our new algebras are relevant to this investigation.

Our work is currently a 3-paper series, of which this is the first.

\subsubsection{%The results of t
This paper}
The results from \cite{quick_ruskuc}, about growth rates
in varieties of classical algebraic structures,
can be presented in a stronger way.
Let $\Sigma$ be a set of identities.
If $\m a$ is an algebra in a language
$\mathcal K$, then say that $\m a$
\emph{realizes} $\Sigma$ 
if there is a way to interpret the function symbols 
occurring in $\Sigma$ as $\mathcal K$-terms
in such a way that each identity in $\Sigma$ holds in $\m a$.
What is really proved in \cite{quick_ruskuc}
is that if $\Sigma_{\text{Grp}}$ is the set of identities
axiomatizing the variety of groups and $\m a$ is a finite
algebra realizing $\Sigma_{\text{Grp}}$,
then $\m a$ has
a logarithmic growth rate if it is perfect
and has a linear growth rate if it is imperfect.
Although the results of \cite{quick_ruskuc} 
are stated for only a few specific varieties of group
expansions, the results hold for 
\emph{any} variety of group expansions.

The main results of this paper
are also best expressed in the terminology
of algebras realizing a set of identities.
Call a term \emph{basic} if it contains at most one 
nonnullary function symbol. 
An identity $s\approx t$ is basic if the terms
on both sides are. This paper is an investigation into
the restrictions imposed on growth rates of finite algebras 
by a set $\Sigma$ of basic identities.
A new concept that emerges from this investigation is 
the notion of a pointed cube term.
If $\Sigma$ is a set of identities in 
a language $\mathcal L$, then an
$\mathcal L$-term $F(x_1,\ldots,x_m)$ is a
\emph{$p$-pointed, $k$-cube term}
for the variety axiomatized by $\Sigma$
if there is a $k\times m$ matrix $M$
consisting of variables and $p$ distinct constant symbols,
with every column of $M$ containing a symbol different 
from $x$, such that 
\begin{equation}\label{cube}
\Sigma\models 
F(M)\approx 
\left(
\begin{matrix}
x\\
\vdots\\
x
\end{matrix}
\right).
\end{equation}
(\ref{cube}) is meant to be a compact representation
of a sequence of $k$ row identities of a special kind.
For example, 
\begin{equation}\label{m}
\Sigma\models 
{m}\left(\begin{matrix}
x&y&y\\
y&y&x
\end{matrix}
\right) \approx
\left(\begin{matrix}
x\\
x
\end{matrix}
\right),
\end{equation}
which is the assertion that 
$\Sigma\models m(x,y,y)\approx x$ and 
$\Sigma\models m(y,y,x)\approx x$,
witnesses that $m(x_1,x_2,x_3)$ is a
$3$-ary, $0$-pointed, $2$-cube term.
The basic identities (\ref{m})
define what is called a \emph{Maltsev term}.
For another example, 
\begin{equation}\label{b}
\Sigma\models 
{B}\left(\begin{matrix}
1&x\\
x&1
\end{matrix}
\right) \approx
\left(\begin{matrix}
x\\
x
\end{matrix}
\right),
\end{equation}
which is the assertion that 
$\Sigma\models B(1,x)\approx x$ and 
$\Sigma\models B(x,1)\approx x$,
witnesses that $B(x_1,x_2)$ is a
$2$-ary, $1$-pointed, $2$-cube term.
As a final example, 
\begin{equation}\label{majority}
\Sigma\models 
{M}\left(\begin{matrix}
y&x&x\\
x&y&x\\
x&x&y
\end{matrix}
\right) \approx
\left(\begin{matrix}
x\\
x\\
x
\end{matrix}
\right),
\end{equation}
which is the assertion that $M$ is a majority term
for the variety axiomatized by $\Sigma$,
witnesses that $M(x_1,x_2,x_3)$ is a
$3$-ary, $0$-pointed, $3$-cube term.

To state our main results, 
let $\Sigma$ be a set of basic identities.
We show that
\begin{enumerate}
\item[(1)]
The growth rate of any partial algebra
can be realized as the growth rate of a total algebra
(Corollary~\ref{partial_cor}).
If the partial algebra is finite, then the total algebra
can be taken to be finite.
\item[(2)]
A function $D\colon \omega\to\omega^+$
arises as the $d$-function
of a countably infinite algebra if and only if
(i)~$D$ is increasing and satisfies
(ii)~$D(0)=0$ or $1$, and (iii)~\mbox{$D(2)>0$}
(Theorem~\ref{big}). 
\item[(3)]
If $\Sigma$ does not entail the existence of a pointed cube
term, then $\Sigma$ imposes no restriction on growth
rates of algebras (Theorem~\ref{nonrestrictive_thm}).
That is, for every algebra $\m a$ there is
an algebra $\m b$ realizing
$\Sigma$ such that $d_{\m b}=d_{\m a}$.
The algebra $\m b$ can be taken to be finite if 
$\m a$ is finite and the set $\Sigma$ involves
only finitely many distinct constants.
\item[(4)]
If $\Sigma$ entails the existence of a $p$-pointed
cube term, $p\geq 1$, then any algebra $\m a$
realizing $\Sigma$ such that $\m a^{p+k-1}$
is finitely generated 
has growth rate that is bounded above
by a polynomial (Theorem~\ref{pointed_polynomial}).
This is a nontrivial restriction.
\item[(5)]
There exist finite algebras
with pointed cube terms whose growth rate is 
asymptotically equivalent to a polynomial 
of any prescribed degree (Theorem~\ref{example}).
\item[(6)]
Any function that arises as the growth rate
of an algebra with a pointed cube term also arises
as the growth rate of an algebra without a pointed cube term
(Theorem~\ref{avoid_thm}).
\end{enumerate}

In addition to these items we give a new proof of 
Kelly's Completeness Theorem for basic identities
(Theorem~\ref{kelly}). 
We give a procedure, based on this theorem, 
for deciding if a finite set of basic
identities implies the existence of a pointed cube term
(Corollary~\ref{decide_cor}). 

\subsubsection{Our second paper, \cite{paper2}}
We investigate growth rates of algebras
with a $0$-pointed $k$-cube term, which we shall just call
a ``$k$-cube term''.
Such terms were first identified in
\cite{bimmvw} in connection with investigations
into constraint satisfaction problems, while an equivalent
type of term was identified independently in
\cite{parallelogram} in connection with 
investigations into compatible relations of algebras.

We show in \cite{paper2}
that if $\m a$ has a $k$-cube term and
$\m a^{k}$ is finitely generated,
then $d_{\m a}(n)\in O(\log(n))$ if $\m a$ is perfect,
while $d_{\m a}(n)\in O(n)$ if $\m a$ is imperfect.
One can strengthen `Big Oh' to `Big Theta'
if $\m a$ is finite.
This extends Wiegold's result (I) for groups 
to a setting that includes, as special cases, any 
finite algebra with a Maltsev term 
(in particular, any finite algebra in a congruence uniform variety)
or any finite algebra with a majority term. 

\subsubsection{Our third paper, \cite{paper3}}
We investigate growth rates of finite solvable algebras.
Our original aim was to show that the only growth rates
exhibited by such algebras are linear or exponential functions.
We do prove this for finite nilpotent algebras 
and we prove it for finite solvable algebras
with a pointed cube term, but the general case 
of a finite solvable algebra without a pointed cube term
remains open.

\section{Preliminaries}\label{prelim_sec}

\subsection{Notation}\label{notation}

$[n]$ denotes the set $\{1,\ldots,n\}$.
A tuple in $A^n$ may be denoted 
$(a_1,\ldots,a_n)$ or $\wec{a}$, and
may be viewed as a function
$\wec{a}\colon [n]\to A$.
A tuple $(a,a,\ldots,a)\in A^n$ with all coordinates equal to $a$
may be denoted $\hat{a}$.
The size of a set $A$, the 
length of a tuple $\wec{a}$, and the length of a string $\sigma$
are denoted $|A|$, $|\wec{a}|$ and $|\sigma|$.
Structures are denoted in bold face font,
e.g. $\m a$, while the universe of a structure
is denoted by the same character 
in italic font, e.g., $A$. 
The subuniverse of $\m a$ generated by a subset $G\subseteq A$
is denoted $\lb G\rb$.

We will use Big Oh notation.
If $f$ and $g$ are real-valued functions defined
on some subset of the real numbers, then 
$f\in O(g)$ and $f=O(g)$ both mean that there are 
positive constants
$M$ and $N$ such that $|f(x)|\leq M|g(x)|$ for all $x>N$.
We write $f\in \Omega(g)$ and $f=\Omega(g)$ to mean that there are 
positive constants
$M$ and $N$ such that $|f(x)|\geq M|g(x)|$ for all $x>N$.
Finally, $f\in \Theta(g)$ and $f=\Theta(g)$ mean that both
$f\in O(g)$ and $f\in \Omega(g)$ hold.

\subsection{Easy estimates}\label{estimates}

\begin{thm}\label{basic_estimates}
Let $\m a$ be an algebra.
\begin{enumerate}
\item[(1)] $d_{\m a^k}(n) = d_{\m a}(kn)$.
\item[(2)] If $\m b$ is a homomorphic image of 
$\m a$, then $d_{\m b}(n) \leq d_{\m a}(n)$.
\item[(3)] If $\m b$ is an expansion of 
$\m a$ (equivalently, if $\m a$ is a reduct of $\m b$), 
then $d_{\m b}(n) \leq d_{\m a}(n)$.
\item[(4)] {\rm (From \cite{quick_ruskuc})}
If $\m b$ is 
the expansion of $\m a$ obtained by adjoining all constants, then
\[
d_{\m a}(n) - d_{\m a}(1)\leq d_{\m b}(n) \leq d_{\m a}(n).
\]
\end{enumerate}
\end{thm}

\begin{proof}
For $(1)$, both 
$d_{\m a^k}(n)$ and $d_{\m a}(kn)$
represent
the number of elements in a smallest size
generating set for $(\m a^k)^n\cong \m a^{kn}$.

For $(2)$, if $\varphi\colon \m a\to \m b$ is surjective and
$G\subseteq A^n$ is a smallest size generating set for $\m a^n$,
then $\varphi(G)$ is a generating set for $\m b^n$.
Hence $d_{\m b}(n)\leq |\varphi(G)|\leq |G| = d_{\m a}(n)$.

For $(3)$, if 
$G\subseteq A^n$ is a smallest size generating set for $\m a^n$,
then $G$ is also a generating set for $\m b^n$.
Hence $d_{\m b}(n)\leq |G| = d_{\m a}(n)$.

For $(4)$, the right-hand
inequality $d_{\m b}(n) \leq d_{\m a}(n)$ follows from $(3)$.
Now let $G\subseteq A^n$ be a smallest size generating set
for $\m b^n$ and let $H\subseteq A$ be a smallest 
size generating set for $\m a$.
For each $a\in H$ let $\hat{a} = (a,a,\ldots,a)\in A^n$ be the associated
constant tuple, and let $\widehat{H}$ be the set of these.
Every tuple of $A^n$ is generated from $G$ by polynomial operations
of $\m a$ acting coordinatewise, hence is generated from 
$G\cup \widehat{H}$ by term operations of $\m a$ acting
coordinatewise. This proves
$
d_{\m a}(n) \leq |G| + |H| = d_{\m b}(n) + d_{\m a}(1),
$ 
from which the left-hand inequality follows.
\end{proof}

The next theorem will not be used later in the paper,
except that in Section~\ref{problems} one should
know that the $d$-function of a finite
algebra is bounded below by a logarithmic function
and above by an exponential function.

\begin{thm}\label{first_bounds}
If $\m a$ is a finite algebra of
more than one element and $n>0$, then 
\[
\lceil \log_{|A|}(n) \rceil\leq d_{\m a}(n) \leq |A|^n
\]
and
\[
\lfloor \log_{|A|}(n) \rfloor\leq h_{\m a}(n) \leq |A|^n.
\]
Hence $d_{\m a}(n), h_{\m a}(n)\in \Omega(\log(n))\cap 2^{O(n)}$.
Moreover,
\begin{enumerate}
\item[(1)] $d_{\m a}(n) \in O(\log(n))$ iff 
$h_{\m a}(n) \in 2^{\Omega(n)}$.
\item[(2)] $d_{\m a}(n) \in O(n)$ iff 
$h_{\m a}(n) \in \Omega(n)$, and 
$d_{\m a}(n) \in \Omega(n)$ iff 
$h_{\m a}(n) \in O(n)$.
\item[(3)] $d_{\m a}(n) \in 2^{\Omega(n)}$ iff
$h_{\m a}(n) \in O(\log(n))$.
\end{enumerate}
\end{thm}

\begin{proof}
It follows from Theorem~\ref{basic_estimates}~(3) that,
among all algebras with universe $A$, the algebra 
with only projection operations for its term operations
has the smallest $d$-function and the
algebra 
with all finitary operations as term operations has
the largest $d$-function. These two algebras are
also extremes for the $h$-function.

If $\m a$ has no nontrivial term operations, then every element of $A^n$
is a required generator, so $d_{\m a}(n) = |A|^n$.
In this case, $h_{\m a}(n)= \lfloor \log_{|A|}(n)\rfloor$ for $n>0$,
since $h$ is the upper adjoint of $d$.

Now assume that $\m a$ has 
all finitary operations as term operations.
The $n$-generated free algebra in the variety generated by $\m a$
is isomorphic to $\m a^{|A|^n}$ (Theorem 3 of \cite{foster}).
Since the largest $n$-generated algebra in this variety
is a power of $\m a$, it is also the largest $n$-generated 
power of $\m a$ in the variety; we obtain that
$h_{\m a}(n) = |A|^n$. 
In this case, $d_{\m a}(n)= \lceil \log_{|A|}(n)\rceil$ for $n>0$,
since $d$ is the lower adjoint of $h$.

The fact that $d_{\m a}$ is the lower adjoint 
of $h_{\m a}$ suggests an asymmetry, in that
\begin{equation}\label{adjoint1}
d_{\m a}(n)\leq k\Longleftrightarrow n\leq h_{\m a}(k),
\end{equation}
relates an upper bound of $d_{\m a}$ to a lower bound
of $h_{\m a}$. But the fact that these functions
are defined between totally ordered sets allows us to
rewrite (\ref{adjoint1}) as
\begin{equation}\label{adjoint2}
h_{\m a}(k) < n\Longleftrightarrow k < d_{\m a}(n),
\end{equation}
which almost exactly reverses condition (\ref{adjoint1}) on 
$d_{\m a}$ and $h_{\m a}$.
Using this fact and the following claim, 
one easily verifies Items (1)--(3).

\begin{clm}
If $f, g\colon [a,\infty)\to \mathbb R$ are increasing
functions that tend to infinity as $x$ tends to infinity,
then 
$\lfloor f(n)\rfloor < d_{\m a}(n)\leq \lceil g(n)\rceil$ 
holds for all large $n$
iff $\lfloor g^{-1}(n)\rfloor \leq h_{\m a}(n) < \lceil f^{-1}(n)\rceil$ 
holds for all large $n$.
\end{clm}

\cproof
Allow ``$\forall^{\infty} N$'' to stand for ``for all large $n$'',
i.e., for ``$(\exists N)(\forall n>N)$''.
We have
\[
\begin{array}{rl}
\forall^{\infty} N(d_{\m a}(n)\leq \lceil g(n)\rceil) 
&\Longrightarrow\;\;
\forall^{\infty} N(n\leq h_{\m a}(\lceil g(n)\rceil))\\
&\Longrightarrow\;\;
\forall^{\infty} N(\lfloor g^{-1}(n)\rfloor \leq 
h_{\m a}(\lceil g(\lfloor g^{-1}(n)\rfloor)\rceil))\\
&\Longrightarrow\;\;
\forall^{\infty} N(\lfloor g^{-1}(n)\rfloor \leq h_{\m a}(n)),
\end{array}
\]
because the monotonicity of $g$ guarantees that 
$\lceil g(\lfloor g^{-1}(n)\rfloor)\rceil\leq n$.
The reverse implication is proved the same way, as are 
both implications in 
$\lfloor f\rfloor < d \Leftrightarrow h < \lceil f^{-1}\rceil$.~\cqed
\end{proof}

Recall that the \emph{free spectrum} of a variety
$\mathcal V$ is the function $f_{\mathcal V}(n):=|F_{\mathcal V}(n)|$
whose value at $n$ is the cardinality of the
$n$-generated free algebra in $\mathcal V$.

\begin{thm}\label{spec1}
If $\m a$ is a nontrivial finite algebra and $f_{\mathcal V}$
is the free spectrum of the variety $\mathcal V = {\mathcal V}(\m a)$,
then 
$
h_{\m a}(n)\leq \log_{|A|}(f_{\mathcal V}(n))
$
for $n>0$.
In particular,
\begin{enumerate}
\item[(1)] 
if $f_{\mathcal V}(n)\in O(n^k)$ for some fixed $k\in\mathbb Z^+$, 
then $d_{\m a}(n)\in 2^{\Theta(n)}$;
\item[(2)] 
if $f_{\mathcal V}(n)\in 2^{O(n)}$,
then $d_{\m a}(n)\in \Omega(n)$. 
\end{enumerate}
\end{thm}

\begin{proof}
Assume that $n>0$.

The algebra $\m a^{h_{\m a}(n)}$ is $n$-generated, hence
a quotient of the $n$-generated free algebra
$\m f_{\mathcal V}(n)$. This proves that 
$|A|^{h_{\m a}(n)}\leq f_{\mathcal V}(n)$, or 
$h_{\m a}(n)\leq \log_{|A|}(f_{\mathcal V}(n))$.

If $f_{\mathcal V}(n)\in O(n^k)$ for some fixed $k\in\mathbb Z^+$, 
then $\log(f_{\mathcal V}(n))\in O(\log(n))$,
hence $h_{\m a}(n)\in O(\log(n))$. 
Theorem~\ref{first_bounds} proves that 
$d_{\m a}(n)\in 2^{\Omega(n)}$ holds when $h_{\m a}(n)$
is bounded like this and that
$d_{\m a}(n)\in 2^{O(n)}$ holds just because $\m a$ is finite,
so $d_{\m a}(n)\in 2^{\Theta(n)}$. 

If $f_{\mathcal V}(n)\in 2^{O(n)}$, then
$\log(f_{\mathcal V}(n))\in O(n)$,
hence $h_{\m a}(n)\in O(n)$. It follows from 
Theorem~\ref{first_bounds}~(2) that 
$d_{\m a}(n)\in \Omega(n)$.
\end{proof}

\begin{cor}\label{abelian_cor}
Let $\m a$ be a nontrivial finite algebra and let $\m b$
be a nontrivial homomorphic image of $\m a^k$ for some $k$.
\begin{enumerate}
\item[(1)] 
If $\m b$ is strongly abelian (or even just strongly rectangular), 
then $d_{\m a}(n)\in 2^{\Theta(n)}$.
\item[(2)] 
If $\m b$ is abelian, then $d_{\m a}(n)\in \Omega(n)$.
\end{enumerate}
\end{cor}

\begin{proof}
For (1), Theorem~5.3
of \cite{strnil} proves that a finite strongly rectangular
algebra generates a variety with free spectrum bounded
above by a polynomial.
By Theorem~\ref{spec1}, $d_{\m a}(n)\in 2^{\Theta(n)}$
in this case. The strong abelian property 
is more restrictive than the strong rectangular
property by Lemma~2.2~(11) of \cite{strnil}.

For (2), any finite abelian algebra generates a variety
$\mathcal V$ 
whose free spectrum satisfies 
$f_{\mathcal V}(n)\in 2^{O(n)}$, according to \cite{berman_mck},
so Theorem~\ref{spec1}~(2) completes the argument.
\end{proof}

Recall that an algebra 
is \emph{affine} if it is polynomially
equivalent to a module. It is known
that $\m a$ is affine iff $\m a$ is 
abelian and has a Maltsev term iff 
$\m a$ is abelian and has a Maltsev polynomial.

\begin{thm}\label{affine_prep}
If $\m a^2$ is a finitely generated affine algebra, 
then $d_{\m a}(n)\in O(n)$. If, moreover, $\m a$ is finite
and has more than one element, then 
$d_{\m a}(n)\in \Theta(n)$.
\end{thm}

\begin{proof}
The theorem is true under the weaker assumption that $\m a$ 
(rather than $\m a^2$) is finitely generated, provided
$\m a$ is a module rather than an arbitrary affine algebra.
To see this,
suppose that $\m m$ is a module generated by a finite subset $G$.
The set of tuples in $\m m^n$
with exactly one nonzero entry, which is taken from $G$, 
is a generating set for $\m m^n$ of size $\leq |G|\cdot n$.
Hence $d_{\m m}(n)\in O(n)$. If, moreover, $\m m$ is finite and has
more than one element, then
Corollary~\ref{abelian_cor}~(2) proves that 
$d_{\m m}(n)\in \Omega(n)$, so $d_{\m m}(n)\in \Theta(n)$.

It now follows from Theorem~\ref{basic_estimates}~(4) 
that if $\m a$ is an algebra that is polynomially
equivalent to a finitely generated module,
then $d_{\m a}(n)\in O(n)$, and $d_{\m a}(n)\in\Theta(n)$
if $\m a$ is finite and nontrivial. Unfortunately, not
every finitely generated affine algebra is polynomially
equivalent to a finitely generated module. 
But if $\m a$
is affine and $\m a^2$ is finitely generated, then
the linearization
$\m a^2/\Delta$ (see \cite{freese-mckenzie} pp.~114)
is also finitely generated and term equivalent to 
a reduct of the underlying module of $\m a$. Hence
when $\m a^2$ is finitely generated, then
$\m a$ is polynomially equivalent to a finitely generated module,
and the conclusions of the theorem hold.
\end{proof}

\section{General growth rates}

\subsection{Growth rates of partial algebras}\label{partial_subsection}

A partial algebra is a set
equipped with a set of partial operations.
A total algebra is considered to be a partial algebra,
but, of course, some partial algebras are not total.

The definitions of functions 
$d_{\m a}$ and $h_{\m a}$ make sense when $\m a$
is a partial algebra,
as does the problem of determining growth rates 
of partial algebras. 
Theorem~\ref{basic_estimates}~(3), which relates the 
growth rate of an algebra to that of a reduct,
holds in exactly the same form if a ``reduct of $\m b$''
is interpreted to mean an algebra $\m a$
with the same universe as $\m b$ whose basic partial operations
are obtained from \emph{some} of the term partial operations of $\m b$
by possibly restricting their domains.

We will learn in this subsection 
that a function arises as the growth rate
of a partial algebra if and only if it arises 
as the growth rate of a total algebra.

\begin{df}\label{one-point}
Let $\m a = \lb A; P\rb$ be a partial algebra
with universe $A$ and a set $P$ of partial operations on $A$.
The \emph{one-point completion} of 
$\m a$ is the total algebra whose
universe is $A_{0}:=A\cup \{0\}$, where
$0$ is some element not in $A$, and whose operations 
$P_0 = \{p_0\;|\;p\in P\}\cup \{\wedge\}$ are defined
as follows.

\begin{enumerate}
\item[(1)] If $p\in P$ is a partial 
$m$-ary operation on $A$ with domain $D\subseteq A^m$,
then the total operation $p_{0}\colon (A_{0})^m\to A_{0}$
is defined by 
\[
p_{0}(\wec{a}) = 
\begin{cases}
p(\wec{a}) & \textrm{if $\wec{a}\in D$;}\\
0 & \textrm{otherwise.}
\end{cases}
\]
\item[(2)]
A meet operation $\wedge$ on $A_{0}$ is defined by
\[
a\wedge b = 
\begin{cases}
a & \textrm{if $a=b$;}\\
0 & \textrm{otherwise.}
\end{cases}
\]
\end{enumerate}
\end{df}

\begin{thm}\label{partial}
Let $\m a$ be a partial algebra of more than one element,
and let $\m a_0$ be its one-point completion.
\begin{enumerate}
\item[(1)] 
Any generating set for $\m a^n$ is a generating set for 
$\m a_0^n$, and 
\item[(2)] Any 
generating set for
$\m a_0^n$ contains 
a generating set for $\m a^n$.
\end{enumerate}
In particular, least size generating sets
for $\m a^n$ and $\m a_0^n$ have the same size, and 
if $\m a^n$ or $\m a_0^n$ have any minimal
generating sets, then they are the same.
\end{thm}

\begin{proof}
In this paragraph we prove (1).
If $G\subseteq A^n$ is a generating set for $\m a^n$, then as a subset
of $\m a_0^n$ it will generate 
(in exactly the same manner) all tuples in $A_0^n$
which have no $0$'s.
If $\wec{z}\in A_0^n$ is an arbitrary 
tuple and $a, b\in A$ are distinct, let
$\wec{z}_a$ and $\wec{z}_b$
be the tuples obtained from $\wec{z}$ by replacing
all $0$'s with $a$ and $b$, respectively. Then $\wec{z}_a, \wec{z}_b
\in A^n$, so they are generated by $G$, and 
$\wec{z} = \wec{z}_a\wedge \wec{z}_b$, so $\wec{z}$ is also generated
by $G$. Hence $G$ generates all of $\m a_0^n$. 

Now we prove (2). Assume that 
$H\subseteq A_0^n$ is a 
generating set for
$\m a_0^n$. If $\wec{a}\in A_0^n$, let $Z(\wec{a})\subseteq [n]$
be the \emph{zero set} of $\wec{a}$, by which we mean the set
of coordinates where $\wec{a}$ is $0$. It is easy to see
that for any basic operation $F$ of $\m a_0$ 
it is the case that 
\begin{equation}\label{generation}
Z(\wec{a}_1)\cup \cdots \cup Z(\wec{a}_m)
\subseteq Z(F(\wec{a}_1,\ldots,\wec{a}_m)),
\end{equation}
since $0$ is absorbing for every basic operation. 
If the right-hand side is empty, then the left-hand
side is empty as well; i.e., tuples with empty zero sets
can be generated only by tuples with empty zero sets.
Said a different way, 
if $H\subseteq A_0^n$ generates $\m a_0^n$,
then $H\cap A^n$ suffices to generate
all tuples in $A^n$.
If you consider how $H$ generates elements of $A^n$
in the algebra $\m a_0^n$, it is clear that $H$ generates
those elements in the algebra $\m a^n$ in exactly the same way, so
$H$ is a generating set for $\m a^n$.
\end{proof}

\begin{cor}\label{partial_cor}
If $\m a$ is a partial algebra and $\m a_0$
is its one-point completion, then
$d_{\m a_0}(n)=d_{\m a}(n)$ for all $n\in\omega$.  \qed
\end{cor}

\subsection{Growth rates of countably infinite algebras}

In this section we characterize the $d$-functions
of countably infinite algebras. We will see that 
there are a few obvious properties that these functions 
have, and that any function
$D\colon \omega\to \omega^+$ that has these
properties may be realized as a $d$-function.

One obvious property of $d$-functions
is that they are increasing: $m\leq n$ implies
$d_{\m a}(m)\leq d_{\m a}(n)$. The $d$-function
of a countably infinite algebra is an increasing function
from the ordered set of natural numbers, $\omega$,
to the ordered set $\omega^+=\omega\cup\{\omega\}=\{0,1,\ldots,\omega\}$,
where $d_{\m a}(n)=\omega$ means that 
$\m a^n$ is not finitely generated.
$d$-functions also have special initial values.
$\m a^0$ is a 1-element algebra, so $\m a^0$ 
is $0$-generated if $\m a$ has a nullary term
and is $1$-generated if $\m a$ has no nullary term.
Thus $d_{\m a}(0)=0$ or $1$, with the cases distinguished
according to whether $\m a$ has a nullary term.
Finally, if $\m a$ has more than one element,
then $d_{\m a}(2) > 0$, since any $0$-generated
subalgebra of $\m a^2$ is contained in the diagonal
and the diagonal is a proper subalgebra 
of $\m a^2$ when $|A|>1$.
We now prove:

\begin{thm}\label{big}
If $D\colon \omega\to \omega^+$ 
\begin{enumerate}
\item[(i)] is increasing,
\item[(ii)] satisfies $D(0)=0$ or $1$, and
\item[(iii)] satisfies $D(2) > 0$,
\end{enumerate}
then there is a countably infinite total algebra $\m a$
such that $d_{\m a}(n)=D(n)$ for all $n\in\omega$.
\end{thm}

\begin{proof}
We construct a partial algebra $\m a$
such that $d_{\m a}(n) = D(n)$ for all $n\in \omega$.
By Corollary~\ref{partial_cor}
the one-point completion of $\m a$ (Definition~\ref{one-point})
will be a total algebra with the same growth rate.

First we describe the universe of our partial algebra.
Start with a countably infinite set $X$. This set
will be a subset of the universe of $\m a$, and its main function
is to ensure that the constructed algebra is infinite.
Next, for any algebra $\m b$,
$d_{\m b}(0)=0$ happens exactly when $\m b$
has a nullary term. Hence if $D(0)=0$ and
we wish to represent $D$
as $d_{\m a}$ for some $\m a$, then we must ensure
that $\m a$ has a nullary term. 
So let $Y=\{y\}$ be a singleton set.
% if $D(0)=0$ and let $Y=\emptyset$ otherwise. 
If we need our algebra
to have a nullary term, we will introduce a term with 
value $y$.
Finally, for each nonzero $n\in\omega$ where $D(n)$ is finite,
let $M^{(n)} = [z_{i,j}^{(n)}]$ 
be an $n\times D(n)$ matrix of elements
such that all entries
of all $M^{(n)}$'s are different from each other and are different
from the elements of $X\cup Y$.
Let $Z = \{z_{i,j}^{(n)}\}$ 
be the set of all entries appearing in these matrices, and 
take $A:=X\cup Y\cup Z$ to be the 
universe of the partial algebra.

If $D(0)=0$, then we introduce a nullary 
operation whose value is $y$. We may introduce
more nullary operations later in the case $D(0)=0$, but if $D(0)=1$ then
we do not introduce any nullary operations throughout the construction.

For each nonzero $n\in\omega$ where $D(n)$ is finite and for each tuple
$\wec{b}\in A^n$, introduce a $D(n)$-ary partial operation
$F_{\wec{b}}$ for which $F_{\wec{b}}(M^{(n)})=\wec{b}$. This means that
$F_{\wec{b}}$ has domain of size $n$, consisting of the $n$ rows of 
$M^{(n)}$, and that $F_{\wec{b}}(z_{i,1}^{(n)},\ldots,z_{i,D(n)}^{(n)})=b_i$ for each
$i=0,1,\ldots,n$. 

It is worth mentioning how to interpret the
instructions of the previous paragraph in the case where $n=1$
and $D(n)=0$. Here $M^{(n)}$ is defined to be a 
$1\times 0$ matrix, and for each $\wec{b}\in A^1=A$ we
are instructed to add a partial operation
$F_{\wec{b}}$ with the property that 
$F_{\wec{b}}(M)=\wec{b}$.  One should view 
$F_{\wec{b}}$ as a nullary partial operation
with range $\wec{b}$. Hence, in the case $(n,D(n))=(1,0)$
we are to add nullary operations naming each element of $A$.
[Consider how one might
interpret the instructions of the previous paragraph 
in the case where $n=2$
and $D(n)=0$, if such were permitted by the assumptions on $D$.
We would be instructed to add
nullary partial operations to $\m a$ 
with range $\wec{b}$ for each $\wec{b}\in A^2$. 
Such nullary operations do not exist for those
$\wec{b}\in A^2$ off of the diagonal, so we would be
unable to adhere to the instructions if we allowed $D(2)=0$. This is
the place in our construction where we make use of
the assumption that $D(2)>0$.]

Our partial algebra is $A$ equipped with all partial operations
of the type described in the previous three paragraphs.

Observe that $d_{\m a}(0) = 0$ iff $\m a$ has a nullary 
term iff $D(0)=0$, so $d_{\m a}(0)=D(0)$. 

Observe that if $D(n)=\omega$ for some $n>0$, then
none of the partial operations has $n$ distinct elements of $A$
in its image. Hence every tuple $\wec{b}\in A^n$ with distinct
coordinates must appear in any generating set for
$\m a^n$. This proves that $d_{\m a}(n)=\omega$ whenever
$D(n)=\omega$.

Observe that if $D(1)=0$, then we have added
nullary operations to $\m a$ naming each element of $A$,
so $d_{\m a}(1)=0$, too.

Now we consider generating sets for $\m a^n$ when
$n>0$ and $D(n)$ is finite and positive. In this case, 
$F_{\wec{b}}(M^{(n)}) = \wec{b}$ whenever 
$\wec{b}\in A^n$, so the columns of $M^{(n)}$
form a generating set of size $D(n)$ for $\m a^n$. 
The following claim will help us to prove that there is no smaller
generating set for $\m a^n$.

\begin{clm}
If $n>0$ and
a subset $G\subseteq A^n$ has fewer than $D(n)$ tuples
whose coordinates are distinct, then the same
is true for $\lb G\rb$.
\end{clm}

\cproof
If the claim is not true, then it must be possible to 
generate in one step
a tuple $\wec{c}\in A^n$ whose coordinates are all distinct
using other tuples, where fewer than $D(n)$ of these other tuples 
have the property that their coordinates are all distinct.
If the partial operation used is some
$F_{\wec{b}}$, $\wec{b}\in A^m$ for some $m$,
and the tuples used to generate are $\wec{x}_1,\ldots, \wec{x}_{D(m)}$,
then the following row equations must be satisfied.
\begin{equation}\label{distinct}
\small{
F_{\wec{b}}(\wec{x}_1,\ldots,\wec{x}_{D(m)})=
F_{\wec{b}}\left( 
\left[
\begin{array}{c}
x_{1,1}\\
\vdots\\
x_{n,1}
\end{array}
\right],
\ldots,
\left[
\begin{array}{c}
x_{1,D(m)}\\
\vdots\\
x_{n,D(m)}\\
\end{array}
\right]
\right) = 
\left[
\begin{array}{c}
c_1\\
\vdots\\
c_n
\end{array}
\right]=\wec{c}.
}
\end{equation}
Considering the definition of $F_{\wec{b}}$, it is clear that
the (distinct!) entries of $\wec{c}$ are among the entries
of $\wec{b}$, so $m=|\wec{b}|\geq |\wec{c}|=n$.
Moreover, the row equations $F_{\wec{b}}(x_{i,1},\ldots,x_{i,D(m)})=c_i$
can be solved in only one way, namely by using the appropriate
row of $M^{(m)}$. This forces all entries of $[x_{i,j}]$ to be distinct.
But this means there are $D(m)$ columns, $\wec{x}_j$,
whose coordinates are distinct,
and we assumed that there were fewer than $D(n)$ such columns.
Altogether this yields that $m\geq n$ and $D(m)<D(n)$, contradicting
the monotonicity of $D(n)$. The claim is proved.
\cqed

\medskip

The claim shows that $d_{\m a}(n)=D(n)$
when $n>0$ and $D(n)$ is finite and positive, 
since a subset $G\subseteq A^n$
of size less than $D(n)$ must have fewer than $D(n)$
tuples whose coordinates are distinct. Such a set cannot generate
$\m a^n$, since the generated subuniverse $\lb G\rb$ contains fewer than
$D(n)$ tuples whose coordinates are distinct while $A^n$
contains infinitely many such tuples.
\end{proof}

The construction in this proof may be modified to
give some information
about $d$-functions of finite algebras.
Namely, suppose that $D\colon \{0,1,\ldots,k\}\to \omega$
%is a finite-valued function that 
is (i)~increasing,
and satisfies (ii)~$D(0)=0$ or $1$, and
(iii)~$D(2)>0$. If one modifies the construction
in the proof by 
omitting the inclusion of the set $X$ in the universe
of $\m a$ and then adding only the partial operations 
that are nullary or of the form $F_{\wec{b}}(M^{(n)})=\wec{b}$
where $n \in \{0, 1, 2, \ldots, k\}$, then the proof shows
that there is an algebra of size 
$|Y\cup Z| = 1+\sum_{j=0}^k j\cdot D(j)$ (finite!) 
such that $d_{\m a}(n)=D(n)$ for 
$n \in \{0, 1, 2, \ldots, k\}$.
Thus there is no special behavior of $d$-functions
of finite algebras on initial segments of $\omega$.

\section{Kelly's Completeness Theorem}\label{Kelly_section}

In Subsection~\ref{Kelly_subsection}
we give a new proof of Kelly's Completeness Theorem
for basic identities. The proof involves
the construction of a model of a set of basic identities.
In Subsection~\ref{V_subsection} we construct a simpler model
by modifying the construction from the Completeness Theorem.
The simpler model is not adequate for proving the
Completeness Theorem, but it is exactly what we need for 
our investigation of growth rates.

\subsection{The Completeness Theorem for basic 
identities}\label{Kelly_subsection}

Let $\mathcal L$ be an algebraic language. Recall that
an $\mathcal L$-term is \emph{basic} if it 
contains at most one nonnullary function symbol.
An $\mathcal L$-identity $s\approx t$ is basic if both
$s$ and $t$ are basic terms. 
If $\Sigma\cup\{\varphi\}$ is a set of basic identities,
then $\varphi$ is a \emph{consequence} of $\Sigma$,
written $\Sigma\models \varphi$,
if every model of $\Sigma$ is a model of $\varphi$.

Let $C$ be the set of constant symbols of $\mathcal L$ and let
$X$ be a set of variables.
The \emph{weak closure of $\Sigma$ in the variables $X$}
is the smallest set $\overline{\Sigma}$
of basic identities containing $\Sigma$ for which 
\begin{enumerate}
\item[(i)] $(t\approx t)\in \overline{\Sigma}$ for all basic
$\mathcal L$-terms $t$ with variables from $X$.
\item[(ii)] If $(s\approx t)\in \overline{\Sigma}$, then 
$(t\approx s)\in \overline{\Sigma}$.
\item[(iii)] If $(r\approx s)\in \overline{\Sigma}$ and
$(s\approx t)\in \overline{\Sigma}$, then 
$(r\approx t)\in \overline{\Sigma}$.
\item[(iv)] 
%(Substitution of variables and constants.) 
If $(s\approx t)\in \overline{\Sigma}$
and $\gamma\colon X\to X\cup C$ is a function, then 
$(s[\gamma]\approx t[\gamma])\in \overline{\Sigma}$, where 
$s[\gamma]$ denotes the basic term obtained from 
$s$ by replacing each variable $x\in X$ with 
$\gamma(x)\in X\cup C$.
\item[(v)] 
%(Replacement of constant subterms.)
If $t$ is a basic $\mathcal L$-term
and $(c\approx d)\in \overline{\Sigma}$ for $c, d\in C$,
then $(t\approx t')\in \overline{\Sigma}$, where
$t'$ is the basic term obtained from $t$
by replacing one occurrence of $c$ with $d$.
\end{enumerate}

These closure conditions may be interpreted as the inference rules
of a proof calculus for basic identities. Therefore, 
write $\Sigma\vdash_X \varphi$ if $\varphi$ belongs
to the weak closure of $\Sigma$ in the variables $X$.
If the set $X$ is large enough, the relation
$\vdash_X$ captures $\models$ for basic identities, as we
will prove in Theorem~\ref{kelly}.
We define $X$ to be \emph{large enough} if 
\begin{enumerate}
\item[(a)] $X$ contains at least 2 variables, 
\item[(b)] $|X|\geq {\rm arity}(F)$  for any
function symbol $F$ occurring in $\Sigma$, and 
\item[(c)] $|X|$ is at least as large as the
number of distinct variables occurring in
any identity in $\Sigma\cup\{\varphi\}$.
\end{enumerate}
Call $\Sigma$ \emph{inconsistent relative to $X$}
if $\Sigma\vdash_X x\approx y$ for distinct 
$x, y\in X$ and large enough $X$.
Otherwise $\Sigma$ is \emph{consistent relative to $X$}.

\begin{thm}[David Kelly, \cite{kelly}] 
\label{kelly}
Let $\Sigma\cup \{\varphi\}$ be a set of basic identities
and $X$ be a set of variables that is large enough.
If $\Sigma$ is consistent relative to $X$, then 
$\Sigma\vdash_X \varphi$
if and only if 
$\Sigma\models \varphi$.
\end{thm}

Kelly's theorem is a natural restriction of 
Birkhoff's Completeness Theorem for equational logic
to the special case of basic identities.
However, it is in general undecidable 
for finite $\Sigma\cup\{\varphi\}$
whether $\Sigma\vdash \varphi$ using Birkhoff's
inference rules, while it is decidable
for basic identities using Kelly's restricted rules.\footnote{%
The reason
that $\Sigma\vdash_X \varphi$ is decidable with
Kelly's inference rules when $\Sigma\cup \{\varphi\}$
is finite is that deciding $\Sigma\vdash_X \varphi$ 
amounts to generating $\overline{\Sigma}$.
If $\mathcal L$ is the language whose function
and constant symbols are those occurring in $\Sigma\cup \{\varphi\}$,
$X$ is a minimal (finite) set of variables that is large enough,
and $\mathcal T$ is defined to be the set of basic
$\mathcal L$-terms in the variables $X$,
then generating $\overline{\Sigma}$ amounts to generating
an equivalence relation on the finite set $\mathcal T$
using Kelly's inference rules.}

In the proof we use a variation of Kelly's Rule (iv):
rather than use functions $\gamma\colon X\to X\cup C$
for substitutions we will use functions
$\Gamma\colon X\cup C\to X\cup C$ whose restriction to $C$ is the identity.
(That is, we replace $\gamma$ with 
$\Gamma := \gamma\cup {\rm id}|_C$.)

\begin{lm}\label{kelly_lm}
If $\Sigma\vdash_X x\approx h$ for some
basic term $h$ in which $x$ does not occur, 
then $\Sigma$ is inconsistent relative to any set $X$
containing a variable other than $x$.
\end{lm}

\begin{proof}
Append to a $\Sigma$-proof of $x\approx h$ 
the formulas $(y\approx h)$ for some $y\in X\setminus \{x\}$ (Rule (iv)); 
$(h\approx y)$ (Rule (ii)); and
$(x\approx y)$ (Rule (iii)).
\end{proof}

\begin{proof}[Proof of Theorem~\ref{kelly}]
Kelly's inference rules are sound, since they are
instances of Birkhoff's inference rules for equational logic.
Hence $\Sigma\vdash_X \varphi$ implies $\Sigma\models \varphi$
for any $X$.

Now assume that $\Sigma{\not{\vdash_X}}\varphi$, 
where $X$ is large enough
and $\Sigma$ is consistent relative to $X$.
We construct a model of $\Sigma\cup \{\neg\varphi\}$ to 
show that $\Sigma{\not{\models}}\varphi$.
Let $\mathcal T$ be the set of basic $\mathcal L$-terms
in the variables $X$,
and let $\equiv$ be the equivalence relation on $\mathcal T$
defined by Kelly provability: i.e., $s\equiv t$ if and only if 
$\Sigma\vdash_X s\approx t$. Write $[t]$ for the $\equiv$-class
of $t$. Now extend $\mathcal T$ to a set
${\mathcal T}_0 = {\mathcal T}\cup \{0\}$
where $0$ is a new symbol, and extend $\equiv$
to this set by taking the equivalence class
of $0$ to be $\{0\}$.

The universe of the model will be the set $M:={\mathcal T}_0/{\equiv}$
of equivalence classes of  ${\mathcal T}_0$ under $\equiv$.
We interpret a constant symbol $c$ as the element
$c^{\m m}:=[c]\in M$. Now let $F$ be an $m$-ary function symbol
for some $m>0$. The natural idea for interpreting $F$
as an $m$-ary operation on this set
is to define $F^{\m m}([a_1],\ldots,[a_m]) = [F(a_1,\ldots,a_m)]$.
However, this does not work, since $F(a_1,\ldots,a_m)$ will not
be a basic term unless all the $a_i$'s 
belong to $X\cup C$.
Nevertheless, we shall follow this idea as far as it takes
us, and when we cannot apply it to assign a value to
$F^{\m m}([a_1],\ldots,[a_m])$ we shall assign the value $[0]$.

Choose and fix a well-order $<$ of the set 
$C$ of constant symbols of $\mathcal L$.
Let $\mathcal I$ be the set of injective partial 
functions $\imath\colon M\to X\cup C$ that satisfy the following
conditions:
\begin{enumerate}
\item[(1)] If 
a class 
$[t]\in M$ in the domain of $\imath$ 
contains a constant 
symbol, $c\in C$, then
$\imath[t] = d$ where $d\in C$ is the least 
element in $[t]\cap C$ under $<$.
\item[(2)] If a class
$[t]$ in the domain of $\imath$
contains a variable, $x\in X$, then
$\imath[t] = x$.
\item[(3)] If a class $[t]$ in the domain of $\imath$
fails to contain a variable
or constant symbol, then $\imath[t]\in X$.
\end{enumerate}
According to Lemma~\ref{kelly_lm},
the consistency of $\Sigma$ implies that any class $[t]$
contains at most one variable, and if $[t]$
contains a constant symbol, then $[t]$ contains
no variable. 
Hence there is no ambiguity in conditions (1) and (2).

If $S\subseteq M$ has size at most $|X|$, then $S$
is the domain of some $\imath\in \mathcal I$.

If $S\subseteq M$ and a class $[t]\in S$
contains a variable $x$, then call $x$ a \emph{fixed} variable
of $S$. Any other variable is an \emph{unfixed} variable of $S$.

Now we define how to interpret an $m$-ary function symbol $F$
as an $m$-ary operation on the set $M$.
Choose any $([a_1],\ldots,[a_m])\in M^m$, 
then choose $\imath\in \mathcal I$
that is defined on $S:=\{[a_1],\ldots,[a_m]\}$.
Note that $f:=F(\imath[a_1],\ldots,\imath[a_m])$ is a basic term, since
it is a function symbol applied to elements of $X\cup C$. We refer
to this term $f$ to define $F^{\m m}([a_1],\ldots,[a_m])$.
\begin{enumerate}
\item[{\rm Case 1.}] 
(The class $[f]$ contains a term $h$ whose only variables
are among the fixed variables of $S$.) 
Define $F^{\m m}([a_1],\ldots,[a_m]) = [f]$.
\item[{\rm Case 2.}] 
($[f]$ contains a variable.)
If $x$ is a variable in $[f]$, 
then $\Sigma\vdash_X f\approx x$. 
Since $\Sigma$ is consistent, Lemma~\ref{kelly_lm} proves that
$x$ must occur in $f$,
i.e., $x=\imath[a_k]$ for some $k$. 
Hence 
\[
\Sigma\vdash_X F(\imath[a_1],\ldots,\imath[a_m])\approx \imath[a_k]
\]
for some $k$. In this case we define 
$F^{\m m}([a_1],\ldots,[a_m])=[a_k]\;(=\imath^{-1}(x)$.)
\item[{\rm Case 3.}] (The remaining cases.)
Define $F^{\m m}([a_1],\ldots,[a_m]) = [0]$.
\end{enumerate}

Before proceeding, we point out that there is overlap
in Cases 1 and 2, but no conflict in the definition
of $F^{\m m}([a_1],\ldots,[a_m])$.
If $[f]$ contains 
a term $h$ whose variables are fixed variables of $S$
and $[f]$ also contains
a variable $x$, then $\Sigma\vdash_X f\approx x$ and
$\Sigma\vdash_X h\approx x$.
The consistency of $\Sigma$ forces $x$ to be a common
variable of $f$ and $h$, and (since only fixed variables
of $S$ occur in $h$) to be a fixed
variable of $S$. In this situation, Case 1 defines
$F^{\m m}([a_1],\ldots,[a_m])=[f]=[x]$ while
Case 2 defines 
$F^{\m m}([a_1],\ldots,[a_m])=\imath^{-1}(x)=[x]$.

\begin{clm}\label{well-defined}
$F^{\m m}\colon M^m\to M$ is a well-defined function.
\end{clm}

\cproof
Choose $([a_1],\ldots,[a_m])\in M^m$
and define $S=\{[a_1],\ldots,[a_m]\}$.
There exist elements of $\mathcal I$ defined on
$S$, because this
set has size $\leq {\rm arity}(F)\leq |X|$.
Suppose that $\imath, \jmath\in \mathcal I$
are both defined on this set. 
Let $f = F(\imath[a_1],\ldots,\imath[a_m])$ and 
$g = F(\jmath[a_1],\ldots,\jmath[a_m])$.
To show that $F^{\m m}([a_1],\ldots,[a_m])$
is uniquely defined it suffices
to show that the same value is assigned
whether we refer to the term $f$ or the term $g$.

In all cases of the definition
of $F^{\m m}([a_1],\ldots,[a_m])$,
the assigned value depends only on 
the term $f = F(\imath[a_1],\ldots,\imath[a_m])
= F(\imath|_S[a_1],\ldots,\imath|_S[a_m])$.
Thus, to complete the proof of Claim~\ref{well-defined},
we may replace both $\imath$ and $\jmath$
by $\imath|_S$ and $\jmath|_S$ and assume that
$\imath$ and $\jmath$ have domain $S$.
Now $\imath$ and $\jmath$ are injective functions 
from $S$ into $X\cup C$, and $\imath[t] = \jmath[t]$
whenever $[t]\in S$ and $[t]$ contains a constant
symbol or a fixed variable of $S$. 
When $\imath[t]\neq \jmath[t]$,
then both are unfixed variables of $S$.
In this situation, there is a function
$\Gamma\colon X\cup C\to X\cup C$ that is the identity
on $C$ and on the fixed variables of $S$
for which $\jmath = \Gamma\circ \imath$.
Hence $f[\Gamma]=g$ and, 
if $h$ is a term whose only variables are fixed
variables of $S$, then $h[\Gamma] = h$.

\begin{enumerate}
\item[{\rm Case 1.}] 
($[f]$ contains a term $h$ whose only variables
are among the fixed variables of $S$.) 
Here
$
\Sigma\vdash_X f=F(\imath[a_1],\ldots,\imath[a_m])\approx h.
$
Append to a $\Sigma$-proof of 
$f\approx h$ the formula
$f[\Gamma]\approx h[\Gamma]$ (Rule (iv)).
Since $f[\Gamma]=g$ and $h[\Gamma]=h$, this is a proof
of $g\approx h$. Next append
$h\approx g$ (Rule (ii)) and
$f\approx g$ (Rule (iii)).
We conclude that $[f]=[g]$, so the value $[f]$ assigned to 
$F^{\m m}([a_1],\ldots,[a_m])$ using $\imath$ is the same as the value
$[g]$ assigned using $\jmath$.
\item[{\rm Case 2.}] 
($[f]$ contains a variable.)
If $x\in X$ is a variable in $[f]$, then 
$x=\imath[a_k]$ for some $k$ and 
$
\Sigma\vdash_X F(\imath[a_1],\ldots,\imath[a_m])\approx \imath[a_k]
$
for this $k$. 
Append to a $\Sigma$-proof of 
$f\approx x$ the formula
$f[\Gamma]\approx x[\Gamma]$ (Rule (iv)).
Since $f[\Gamma]=g$ and $x[\Gamma]=\jmath[a_k]$, 
we conclude that 
$
\Sigma\vdash_X F(\jmath[a_1],\ldots,\jmath[a_m])\approx \jmath[a_k]
$
for the same $k$. Whether we use $\imath$ or $\jmath$ we get
$F^{\m m}([a_1],\ldots,[a_m])=[a_k]$.
\item[{\rm Case 3.}] (The remaining cases.)
In Case 1 we showed that $[f]=[g]$ while in Case 2 we showed
that if $x$ is a variable in $[f]$, then 
$x[\Gamma]$ is a variable in $[g]$; together these show that
if $[g]$ does not contain a variable nor a term
whose only variables are among the fixed variables of $S$,
then the same is true of $[f]$. This argument
works with $f$ and $g$ interchanged, so the remaining cases are
those where both $[f]$ and $[g]$ contain no variables nor terms
whose only variables are among the fixed variables of $S$.
Whether we use $\imath$ or $\jmath$, we get
$F^{\m m}([a_1],\ldots,[a_m]) = [0]$.
\end{enumerate}
\cqed

\medskip

$\m m$ is defined.
We now argue that $\m m$ is a model of $\Sigma$.
Choose an identity $(s\approx t)\in\Sigma$.
If $s$ is an $n$-ary function symbol $F$
followed by a sequence $\alpha\colon [n]\to X\cup C$ 
of length $n$ consisting of variables and constant symbols, 
then let $F[\alpha]$ be an abbreviation for $s$.
If $s$ is a variable or constant symbol, 
then $s$ determines a function 
$\alpha\colon [1]\to X\cup C\colon 1\mapsto s$,
so abbreviate $s$ by $\diamondsuit[\alpha]$.
We will, in fact, write $s$ as $F[\alpha]$ in either
case, but will remember that $F$ may equal 
the artificially introduced symbol $\diamondsuit$.
The identity $s\approx t$
takes the form $F[\alpha]\approx G[\beta]$.

A valuation in $\m m$ is a function $v\colon X\cup C \to M$
satisfying 
$v(c)=c^{\m m}=[c]$ for each $c\in C$. To show that
$\m m$ satisfies $F[\alpha]\approx G[\beta]$
we must show that 
$F^{\m m}[v\circ \alpha] = G^{\m m}[v\circ \beta]$
for any valuation $v$.
Choose $\imath\in \mathcal I$ that is defined on the
set $\im(v\circ \alpha)\cup \im(v\circ \beta)$.
This is possible, since we assume that $|X|$
is at least as large as the number of distinct 
variables in the identity
$F[\alpha]\approx G[\beta]\in \Sigma$.
The values of $F^{\m m}[v\circ \alpha]$ 
and $G^{\m m}[v\circ \beta]$
are defined in reference to the terms 
$f:=F[\imath\circ v\circ \alpha]$
and 
$g:=G[\imath\circ v\circ \beta]$
respectively. 

\begin{clm}\label{f=g}
$[f]=[g]$.
\end{clm}

\cproof
Observe that $(\imath\circ v)(c) = \imath[c] = d$,
where $d\in C$ is the $<$-least constant symbol in the class
$[c]$.
If $\Gamma\colon X\cup C\to X\cup C$ is a function
that agrees with $(\imath\circ v)$ on the variables in
$\im(\alpha)\cup \im(\beta)$,
but is the identity on $C$,
then applications of Rule (v) show that
$\Sigma\vdash_X F[\imath\circ v\circ \alpha]\approx
F[\Gamma\circ \alpha]$ and
$\Sigma\vdash_X G[\imath\circ v\circ \beta]\approx
G[\Gamma\circ \beta]$.
From  Rule~(iv), the fact that
$(F[\alpha]\approx G[\beta])\in\Sigma$
implies that
$\Sigma\vdash_X F[\Gamma\circ \alpha]\approx G[\Gamma\circ \beta]$.
Hence 
\[
\Sigma\vdash_X f=F[\imath\circ v\circ \alpha]\approx 
F[\Gamma\circ \alpha]\approx
G[\Gamma\circ \beta]\approx
G[\imath\circ v\circ \beta]=g,
\] 
from which we get $[f]=[g]$.
\cqed

\medskip

We conclude the argument that $\m m$ satisfies 
$F[\alpha]\approx G[\beta]$
as follows.
\begin{enumerate}
\item[{\rm Case 1.}] 
($[f]=[g]$ contains a term $h$ whose only variables
are among the fixed variables of $S$.) 
In this case 
$F^{\m m}[v\circ \alpha]=[f]=[g]=G^{\m m}[v\circ \beta]$.
\item[{\rm Case 2.}] 
($[f]=[g]$ contains a variable.)
If $[f]=[x]=[g]$, then 
$F^{\m m}[v\circ \alpha]=\imath^{-1}(x)=G^{\m m}[v\circ \beta]$. 
\item[{\rm Case 3.}] (The remaining cases with $[f]=[g]$.)
$F^{\m m}[v\circ \alpha]=[0]=G^{\m m}[v\circ \beta]$.
\end{enumerate}

To complete the proof 
of the theorem we must show that $\m m$ does not
satisfy $\varphi$. Suppose $\varphi$
has the form $F[\alpha]\approx G[\beta]$.
Let $v$ be the canonical valuation 
\[
X\cup C \to M\colon x\mapsto [x], c\mapsto [c].
\]
Choose $\imath\in\mathcal I$
that is defined on $\im(v\circ \alpha)\cup \im(v\circ \beta)$.
It follows from the definitions that 
$\imath\circ v\colon X\cup C\to X\cup C$ fixes every variable 
in $\im(\alpha)\cup \im(\beta)$, 
while $(\imath\circ v)(c) = d$ is the $<$-least constant symbol
in the class of $c$.
If $\Gamma$ is the identity
function on $X\cup C$, then
just as in the proof of 
Claim~\ref{f=g}, we obtain 
$\Sigma\vdash_X f=F[\imath\circ v\circ \alpha]\approx 
F[\Gamma\circ \alpha] = F[\alpha]$ and 
$\Sigma\vdash_X g=G[\imath\circ v\circ \beta]\approx 
G[\Gamma\circ\beta]=G[\beta]$.
Now $[f]$ contains a term $h:=F[\alpha]$ 
whose only variables are among the fixed variables of $S=\im(\alpha)$,
so we are in Case 1 of the definition of $F^{\m m}$.
Hence $F^{\m m}(v\circ \alpha) = [F[\alpha]]$, and similarly
$G^{\m m}(v\circ \beta) = [G[\beta]]$.
Part of our assumption about $\varphi=(F[\alpha]\approx G[\beta])$ 
is that $\Sigma\not\vdash_X \varphi$, so 
$[F[\alpha]]$ and $[G[\beta]]$ are distinct elements
of $M$. Therefore, $v$ witnesses that
$\m m$ does not satisfy $\varphi$.
\end{proof}

Theorem~\ref{kelly} establishes that if $X$ and $Y$ are two sets of variables
that are large enough, then 
$\Sigma\vdash_X\varphi$ holds iff 
$\Sigma\vdash_Y\varphi$, and hence 
$\Sigma$ is consistent relative to
$X$ if and only if it is consistent relative to $Y$.
Now that the theorem is proved, we drop
the subscript in $\vdash_X$ and the phrase ``relative to $X$''
when writing about provability.

\subsection{The model $\m v$}\label{V_subsection}

Later in the paper we prove theorems about
finite algebras realizing a set $\Sigma$ of basic identities.
For this, we need to be able to construct finite
models of $\Sigma$. The model constructed in Theorem~\ref{kelly}
may be infinite, so we explain how to produce
finite models.

\begin{df}\label{V}
Let $\Sigma$ be a set of basic identities in a language $\mathcal L$
whose set of constant symbols is $C$.
Let $Y$ be a set of variables, $z$ a variable not in $Y$,
and $X$ a large enough set of variables containing
$Y\cup \{z\}$.
Let $V$ be the subset of the model $\m m$ constructed in the proof
of Theorem~\ref{kelly} consisting of 
\[
\{[y]\;|\;y\in Y\}\cup \{[c]\;|\;c\in C\}\cup \{[0]\}.
\]
Write $[Y]$ for $\{[y]\;|\;y\in Y\}$ and 
$[C]$ for $\{[c]\;|\;c\in C\}$.

As in the proof of Theorem~\ref{kelly}, let $<$ be a well-ordering
of $C$. 
If $F$ is an $m$-ary function symbol of $\mathcal L$
and $([a_1],\ldots,[a_m])\in V^n$, then 
let 
\begin{enumerate}
\item[(1)] $\imath [a_k] = d$ if $[a_k]\in [C]$ and $d$ is the 
$<$-least element of $[a_k]\cap C$,
\item[(2)] $\imath [a_k] = y$ if $[a_k]=[y]\in [Y]$, and 
\item[(3)] $\imath [a_k] = z$ if $[a_k]=[0]$.
\end{enumerate}
Define
$
F^{\m v}([a_1],\ldots,[a_m]) = [t] 
$
if there exists $t\in Y\cup C$ such that
\begin{equation}\label{V_def_term}
\Sigma\vdash F(\imath [a_1],\ldots,\imath [a_m]) \approx t,
\end{equation}
and define $F^{\m v}([a_1],\ldots,[a_m]) = [0]$
if there is no such $t$.

$\m v$ is the algebra with universe $V$ 
equipped with all operations of the form
$F^{\m v}$.
\end{df}

\begin{thm}\label{small_models}
$\m v$ is a model of $\Sigma$. 
\end{thm}

\begin{proof}
Let $F[\alpha]\approx G[\beta]$ be an identity 
in $\Sigma$, and let $v\colon X\cup C\to V$ be a valuation.
We must show that 
$F^{\m v}(v\circ \alpha) = G^{\m v}(v\circ \beta)$.

The function $v$ is also a valuation
in $\m m$, because $V\subseteq M$. 
Since $\m m$ is a model of $\Sigma$, we get
$F^{\m m}[v\circ \alpha] = G^{\m m}[v\circ \beta]$.
Choose $\imath\in \mathcal I$ defined on the set
$\im(v\circ\alpha)\cup \im(v\circ\beta)$
such that $\imath[0]=z$, if $[0]$ is in this set.
Let $f = F[\imath\circ v\circ \alpha]$ and 
$g = G[\imath\circ v\circ \beta]$. 
As in the proof of Claim~\ref{f=g}, 
if $\Gamma\colon X\cup C\to X\cup C$ 
is the identity on $C$ and agrees with 
$\imath\circ v$ on the variables
in $\im(\alpha)\cup \im(\beta)$, then 
$\Sigma\vdash f \approx
F[\Gamma\circ \alpha]\approx G[\Gamma\circ \beta] \approx g$.

The term 
$F(\imath [a_1],\ldots,\imath [a_m])$
from line (\ref{V_def_term})
is none other than $f$. 
$F^{\m v}[v\circ \alpha] = [t]$ for some $t\in Y\cup C$
if and only if $\Sigma\vdash f = 
F(\imath [a_1],\ldots,\imath [a_m]) \approx t$.
But since $\Sigma\vdash f\approx g$ 
we also get 
$G^{\m v}[v\circ \beta]=[t]$. This shows that 
$F^{\m v}[v\circ \alpha]$ and $G^{\m v}[v\circ \beta]$
are equal when at least one of them is not $[0]$.
Of course, they are also equal when both of them
equal $[0]$, so $F^{\m v}(v\circ \alpha) = G^{\m v}(v\circ \beta)$.
\end{proof}

\begin{cor}\label{small_models_cor}
If $\Sigma$ is a consistent set of basic identities
in a language whose set of constant symbols is $C$, then 
$\Sigma$ has models of every cardinality strictly exceeding
$|C|$.
\end{cor}

\begin{proof}
Vary the size of $Y$ in the definition of $\m v$,
and use Theorem~\ref{small_models}.
\end{proof}

Corollary~\ref{small_models_cor}
is close to the best possible result 
about sizes of models of a set of basic identities, 
as the next example shows.

\begin{exmp}\label{example_constants}
Let $C$ be a set of constant symbols
and let ${\mathcal B} = \{B_{c,d}\;|\;c, d\in C, c\neq d\}$
be a set of binary function symbols. Let 
\[
\Sigma = \{B_{c,d}(c,x)\approx x, B_{c,d}(d,x)\approx d\;|\;c, d\in C, c\neq d\}.
\]
$\Sigma$ is a consistent set of basic identities,
since if $A$ is any set containing $C$ we can interpret
each $c\in C$ in $A$ as itself and each
$B_{c,d}$ on $A$ by letting $B_{c,d}^{\m a}(c,y)=y$
and $B_{c,d}^{\m a}(x,y)=x$ if $x\neq c$.

If $\m m$ is any model of $\Sigma$ and $c^{\m m} = d^{\m m}$
for some $c, d\in C$, then the identity function
$B^{\m m}_{c,d}(c^{\m m},x)$ equals the constant function
$B^{\m m}_{c,d}(d^{\m m},x)$, so $|M|=1$. Thus elements
of $C$ must have distinct interpretations in any
nontrivial model of $\Sigma$, implying that 
nontrivial models have size at least 
$|C|$.
\end{exmp}

\section{Restrictive $\Sigma$}\label{examples}

Call a set $\Sigma$ of basic identities
\emph{nonrestrictive} 
if, whenever $\m a$ is an 
algebra,
there is an 
algebra $\m b$ realizing $\Sigma$
such that $d_{\m b}(n)=d_{\m a}(n)$. 
Otherwise $\Sigma$ is \emph{restrictive}. 

Call $\Sigma$ \emph{nonrestrictive for finite algebras}
if, whenever $\m a$ is a finite
algebra, 
there is a \underline{finite}
algebra $\m b$ realizing $\Sigma$
such that $d_{\m b}(n)=d_{\m a}(n)$. 
Otherwise $\Sigma$ is 
\emph{restrictive for finite algebras}.

It is possible for $\Sigma$ to be nonrestrictive,
yet restrictive for finite algebras. The set
$\Sigma$ from Example~\ref{example_constants}
has this property when the set of constants
is infinite (cf.\ Remark~\ref{realize}). 
But the concepts defined in the two preceding
paragraphs are close enough that the arguments of this section
apply equally well to both of them. We will see that
if only finitely many constant symbols appear
in $\Sigma$, then $\Sigma$ is restrictive if and only
if it is restrictive for finite algebras.
Both are equivalent to 
the property that $\Sigma$ entails the existence of a 
pointed cube term.

Recall from the introduction that an $m$-ary, 
$p$-pointed, $k$-cube term 
for the variety axiomatized by $\Sigma$
is an
$m$-ary term $F(x_1,\ldots,x_m)$ 
for which there is a $k\times m$ matrix 
$M=[y_{i,j}]$ of variables and 
constant symbols, where 
every column contains a symbol different from $x$,
such that $\Sigma$ proves the identities
\[
F(\wec{y}_1,\ldots,\wec{y}_m)=
F\left(
\left[\begin{array}{c}
y_{1,1}\\
\vdots\\
y_{k,1}\\
\end{array}\right],
\cdots,
\left[\begin{array}{c}
y_{1,m}\\
\vdots\\
y_{k,m}\\
\end{array}\right]\right)
\approx
\left[\begin{array}{c}
x\\
\vdots\\
x\\
\end{array}\right].
\]
In any nontrivial situation the parameters
are constrained by $m, k\geq 2$ and $p\geq 0$.

In Subsection~\ref{pointed_subsection}
we prove that if $\Sigma$ is restrictive, then
it entails the existence of a pointed
cube term. The converse is proved in 
Subsection~\ref{processing}, by showing 
that an algebra with a pointed cube term 
whose $d$-function assumes only finite values
has growth rate that is bounded above by a polynomial.
In particular, it is shown that
a finite algebra $\m a$ with a $1$-pointed $k$-cube term
satisfies $d_{\m a}(n)\in O(n^{k-1})$.
In Subsection~\ref{poly_growth} we give an example
of a $3$-element algebra with a $1$-pointed $k$-cube term whose
growth rate satisfies $d_{\m a}(n)\in \Theta(n^{k-1})$, showing
that the preceding estimate is sharp.
In Subsection~\ref{avoid} we show that any function
$D\colon \mathbf{Z}^+\to\mathbf{Z}^{\geq 0}$
that occurs as the $d$-function 
of an algebra with a pointed cube term also
occurs as the $d$-function 
of an algebra that does not have a pointed cube term.
In Subsection~\ref{exp} we describe one way of
showing that an algebra has an exponential growth rate,
and we use it to exhibit a variety containing a chain of finite
algebras $\m a_1\leq \m a_2\leq \cdots$, each one
a subalgebra of the next, where 
$\m a_i$ has logarithmic growth when $i$ is odd
and exponential growth when $i$ is even.

\subsection{Restrictive $\Sigma$ forces a 
pointed cube term}\label{pointed_subsection}

Let $\Sigma$ be a set of basic identities 
in a language $\mathcal L$ 
whose set $C$ of constant symbols is finite.
Given an algebra $\m a$ in a language disjoint from $\Sigma$
we construct another algebra
$\m a_{\Sigma}$ which realizes $\Sigma$, where 
$\m a_{\Sigma}$ is finite if $\m a$ is.

For the first step, let 
$[C] = \{[c_1],\ldots,[c_p]\}$ be the 
same set of equivalence
classes denoted by $[C]$ in Definition~\ref{V}.
These classes represent the different 
$\Sigma$-provability classes of constant symbols.
If there are $p$ such classes,
then apply the one-point completion construction 
of Subsection~\ref{partial_subsection}
$p+1$ times to $\m a$ to produce 
a sequence $\m a$, $\m a_{z_1}$, $\m a_{z_1, z_2}$, \ldots,
ending at $\m a_{z_1,\ldots,z_p,0}$.
This is an algebra whose universe
is the disjoint union of $A$ and $\{z_1,\ldots,z_p, 0\}$.

$\m a_{\Sigma}$ will be an expansion of 
$\m a_{z_1,\ldots,z_p,0}$ obtained by merging the latter algebra
with the model $\m v$ introduced in 
Definition~\ref{V}.
Let $Y$ be a set of variables satisfying
$|Y|=|A|$,
and let $[Y] = \{[y]\;|\;y\in Y\}$ be the 
set of equivalence
classes also denoted by $[Y]$ in Definition~\ref{V}.
The universe of $\m v$ is the disjoint union
$V = [Y]\cup [C]\cup \{[0]\}$.

Let $\varphi\colon [Y]\to A$ be a bijection.
Extend this to a bijection from $V=[Y]\cup [C]\cup\{[0]\}$
to $A\cup \{z_1,\ldots,z_p\}\cup\{0\}$
by defining $\varphi([c_i]) = z_i$ and $\varphi([0])=0$.
Now $\varphi$ is a bijection from the 
universe of $\m V$ to the universe of $\m a_{z_1,\ldots,z_p,0}$.
Use this bijection to transfer the operations
of $\m v$ over to $\m a_{z_1,\ldots,z_p,0}$ to create $\m a_{\Sigma}$. 
Specifically, the interpretation of the constant symbol
$c_i$ in $\m a_{\Sigma}$ will be $z_i$, and 
if $F$ is an $m$-ary function symbol of $\mathcal L$,
then 
\begin{equation}\label{Ops}
F^{\m a_{\Sigma}}(x_1,\ldots,x_m):=
\varphi(F^{\m v}(\varphi^{-1}(x_1),\ldots,\varphi^{-1}(x_m)))
\end{equation}
will be the interpretation of the symbol $F$
in $\m a_{\Sigma}$.
$\m a_{\Sigma}$ 
is the expansion of 
$\m a_{z_1,\ldots,z_p,0}$ by all constant operations 
$c_i^{\m a_\Sigma}$
and
all operations
of the form (\ref{Ops}). Under this definition
the function $\varphi$ is an isomorphism from 
$\m v$ to the $\mathcal L$-reduct 
of $\m a_{\Sigma}$.

\begin{lm}\label{nonrestrictive}
Let $\m a$ be an algebra with more than
one element and let 
$\Sigma$ be a set of basic identities
involving finitely many constant symbols.
Let $\mathcal V$ be the variety axiomatized by $\Sigma$.
The following statements about an integer $k\geq 2$ are equivalent.
\begin{enumerate}
\item[(1)]
%\item[$(1)^{\hphantom{\prime}}$] 
$\mathcal V$ has a pointed $k$-cube term.
\item[(2)]
%\item[$(2)^{\hphantom{\prime}}$] 
For any $n\geq k$, $A^n_{\Sigma}\setminus A^n$ 
generates $\m a^n_{\Sigma}$.
\item[(3)]
%\item[$(2)^{\prime}$] 
For any $n\geq k$, there is a generating set
$G(n)$ of $\m a^n_{\Sigma}$ such that 
$G(n)\cap A^n$ does not generate $\m a^n$.
\item[(4)]
%\item[$(3)^{\hphantom{\prime}}$] 
$\mathcal V$ has a pointed $k$-cube term
of the form $F(x_1,\ldots,x_m)$, 
where $m\geq 2$, $F$ is a function
symbol occurring in $\Sigma$, and 
the variables $x_1,\ldots,x_m$ are distinct.
\end{enumerate}
\end{lm}

\begin{proof}

[$(1)\Rightarrow(2)$] 
Let $F(x_1,\ldots,x_m)$ be a 
pointed $k$-cube term
of the variety axiomatized by $\Sigma$.
There is a $k\times m$ matrix $M=[y_{i,j}]$ 
of variables and 
$\mathcal L$-constant symbols, where 
every column contains a symbol different from $x$,
such that $\Sigma$ proves the identities
\begin{equation}\label{Pidentities}
F(\wec{y}_1,\ldots,\wec{y}_m)=
F\left(
\left[\begin{array}{c}
y_{1,1}\\
\vdots\\
y_{k,1}\\
\end{array}\right],
\cdots,
\left[\begin{array}{c}
y_{1,m}\\
\vdots\\
y_{k,m}\\
\end{array}\right]\right)
\approx
\left[\begin{array}{c}
x\\
\vdots\\
x\\
\end{array}\right].
\end{equation}

Choose any tuple $\wec{a}\in A^n_{\Sigma}$.
Using the 
row identities of (\ref{Pidentities}),
solve the equation
$F(\wec{b}_1,\ldots,\wec{b}_m)=\wec{a}$
for the $\wec{b}_i$'s,
row by row, according to the following rules.
In the $i$-th row, 
\begin{enumerate}
\item[(a)] if $y_{i,j}=x$, then let $b_{i,j}=a_i$.
\item[(b)] if $y_{i,j} = c_r$ is a constant symbol, 
then let $b_{i,j}=z_r$ be its interpretation in $\m a_{\Sigma}$.
\item[(c)] if $y_{i,j}$ is a variable different from $x$, 
then let $b_{i,j}=0$.
\end{enumerate}
Under these choices, $\wec{b}_i\in A_{\Sigma}^n$ for all $i$
and $F(\wec{b}_1,\ldots,\wec{b}_m)=\wec{a}$.
Moreover, since each column $\wec{y}_i$
in (\ref{Pidentities}) has a symbol
different from $x$, it follows from (a)--(c) that each
$\wec{b}_i$ has a coordinate value that is in the set 
$\{z_1,\ldots,z_p,0\}$. Hence  $\wec{b}_i\in A_{\Sigma}^n\setminus A^n$
for all $i$. This shows that the arbitrarily chosen
tuple $\wec{a}\in A^n_{\Sigma}$ lies in the subalgebra
$\m a^n_{\Sigma}$ that is generated by 
$A_{\Sigma}^n\setminus A^n$.

[$(2)\Rightarrow(3)$] 
Let $G(n) = A^n_{\Sigma}\setminus A^n$.

[$(3)\Rightarrow(4)$] 
Let $G(n)$ be the generating set for $\m a^n_{\Sigma}$
that is guaranteed by Item~(3). Since $G(n)\cap A^n$ does not generate
$\m a^n$, it follows from Theorem~\ref{partial}
that $(A^n_{\Sigma}\setminus A^n)\cup G(n)$ is not a generating
set for $\m a^n_{z_1,\ldots,z_p,0}$.
Let $S$ be the proper subuniverse of 
$\m a^n_{z_1,\ldots,z_p,0}$
that is generated by
$(A^n_{\Sigma}\setminus A^n)\cup G(n)$.

Since $S$ contains $G(n)$, which generates $\m a^n_{\Sigma}$,
and contains the interpretations of the $\mathcal L$-constants,
it cannot be closed under the interpretations 
of the function symbols of $\mathcal L$. 
Hence there is a tuple $\wec{a}\notin S$, an $m$-ary function symbol $F$, 
and $m$ tuples $\wec{b}_1,\ldots, \wec{b}_m\in S$ 
such that $F^{\m a_{\Sigma}}(\wec{b}_1,\ldots,\wec{b}_m)=\wec{a}$.
Necessarily $\wec{a}\in A^n$.

Using the isomorphism $\varphi$ from $\m v$ to 
the $\mathcal L$-reduct of $\m a_{\Sigma}$, we obtain 
that there is a tuple $\wec{y} = \varphi^{-1}(\wec{a})\in 
\varphi^{-1}(A^n)=[Y]^n$ and tuples 
$\wec{v}_i = \varphi^{-1}(\wec{b}_i)\neq \wec{y}$ such that 
$
F^{\m v}(\wec{v}_1,\ldots,\wec{v}_m)=\wec{y}.
$
Since $\wec{v}_1\neq \wec{y}$, there is a coordinate $\ell$
where these tuples differ. In the $\ell$-th coordinate we have
$F^{\m v}([v_{{\ell},1}],\ldots,[v_{{\ell},m}])=[y_{\ell}]$ for 
some variable $y_{\ell}\in Y$ and some elements 
$v_{\ell,j}\in Y\cup C\cup \{0\}$ with $[v_{{\ell},1}]\neq [y_{\ell}]$.
By the definition of $\m v$, 
\begin{equation}\label{cube1}
\Sigma\vdash F(\imath [v_{\ell,1}],\ldots,\imath [v_{\ell,m}])\approx y_{\ell},
\end{equation}
where $\imath [v_{\ell,j}]=v_{\ell,j}$ when $v_{\ell,j}\in Y$,
$\imath[v_{\ell,j}]$ is a constant $\Sigma$-provably
equivalent to $v_{\ell,j}$ when $v_{\ell,j}\in C$,
and $\imath [v_{\ell,j}]=z$ is a variable not in $Y$ when 
$v_{\ell,j}=0$. 
Since $[v_{\ell,1}]\neq [y_{\ell}]$, we have $v_{\ell,1}\neq y_{\ell}$.
After renaming variables, (\ref{cube1}) can be rewritten as
\[
\Sigma\vdash F(m_{1,1},\ldots,m_{1,m})\approx x,
\]
where each $m_{1,j}$ is a variable or constant and $m_{1,1}\neq x$.
Similarly, for each $i$, the fact that $\wec{v}_i\neq \wec{y}$
produces an identity
\[
\Sigma\vdash F(m_{i,1},\ldots,m_{i,m})\approx x,
\]
where each $m_{i,j}$ is a variable or constant and $m_{i,i}\neq x$.
Thus, it is a consequence of $\Sigma$ that the row identities
of 
\[
F([m_{i,j}])\approx
\left(\begin{matrix}
x\\
\vdots\\
x
\end{matrix}
\right)
\]
hold. 
Since the diagonal elements of $[m_{i,j}]$ are not $x$,
these identities make
$F$ a pointed cube term for $\mathcal V$.

[$(4)\Rightarrow(1)$] 
This is a tautology.
\end{proof}

One consequence of Lemma~\ref{nonrestrictive}
is a procedure to decide if a strong
Maltsev condition involving only basic identities
implies the existence of a pointed cube term.

\begin{cor}\label{decide_cor}
A strong Maltsev condition defined by a set 
$\Sigma$ of basic identities entails the existence
of a pointed $k$-cube term if and only if 
it is possible to prove from $\Sigma$ that
some term of the form 
$F(x_1,\ldots,x_m)$ is a pointed 
$k$-cube term, where $m\geq 2$, $F$ is a function
symbol occurring in $\Sigma$, and 
the variables $x_1,\ldots,x_m$ are distinct.
\end{cor}

\begin{proof}
A strong Maltsev condition defined by a set 
$\Sigma$ of identities 
entails the existence
of a pointed $k$-cube term if and only if 
the variety axiomatized by $\Sigma$ has a 
pointed $k$-cube term, so the 
corollary follows from 
Lemma~\ref{nonrestrictive} (1)$\Leftrightarrow$(3).
\end{proof}

That the property in the theorem statement
can be decided follows from Theorem~\ref{kelly}.

The next result is the main one of the subsection.

\begin{thm}\label{nonrestrictive_thm}
Let $\Sigma$ be a set of basic identities involving 
finitely many constant symbols. If $\Sigma$
does not entail the existence of a pointed cube term,
then $\Sigma$ is nonrestrictive (and also
nonrestrictive for finite algebras).
\end{thm}

\begin{proof}
Recall that ``$\Sigma$
is nonrestrictive'' means that 
for every algebra $\m a$ there is a
algebra $\m b$ realizing $\Sigma$
such that $d_{\m b}=d_{\m a}$, ``$\Sigma$
is restrictive'' means the opposite.

Assume that $\Sigma$ fails to 
entail the existence of a pointed cube term.
Choose $\m a$ arbitrarily and let $\m b = \m a_{\Sigma}$.
$\m b$ realizes $\Sigma$ because $\m v$
is a reduct of $\m b$ and a model of $\Sigma$. 
We argue that $d_{\m b}=d_{\m a}$.

Choose a generating set $G$ for $\m b^n$ such that
$|G|=d_{\m b}(n)$. By 
Lemma~\ref{nonrestrictive} (1)$\Leftrightarrow$(3)
we get that $G\cap A^n$ is a generating set for
$\m a^n$, so $d_{\m a}(n)\leq |G\cap A^n|\leq d_{\m b}(n)$.

Now choose a 
generating set $H$ for $\m a^n$ such that 
$|H|=d_{\m a}(n)$. Repeated use of 
Theorem~\ref{partial}~(1) shows that $H$
generates $\m a^n_{z_1,\ldots,z_p,0}$, hence also generates
$\m a^n_{\Sigma}=\m b$. This shows that
$d_{\m b}(n)\leq |H|=d_{\m a}(n)$.
\end{proof}

\begin{remark}\label{realize}
In the third paragraph of this section we stated
that the set
$\Sigma$ from Example~\ref{example_constants}
is nonrestrictive,
yet restrictive for finite algebras when
$\Sigma$ involves infinitely many constants.
Here we explain why this remark is true,
and also explain
to what degree we may remove the assumption
of finitely many constants in 
Theorem~\ref{nonrestrictive_thm}.

Let $\Sigma$ be as in Example~\ref{example_constants}
with $C$ an infinite set of constants. Let $\m a$ be any
finite algebra. There is no finite $\m b$
that realizes $\Sigma$, hence none that realizes
$\Sigma$ and satisfies 
$d_{\m b}=d_{\m a}$, since any nontrivial
model of $\Sigma$ has cardinality at least $|C|$.

On the other hand, $\Sigma$ does not entail the existence of a pointed
cube term. Without attempting to give the full argument
for this, we indicate only that if $\Sigma$ entailed
the existence of a pointed cube term, then 
(i)~one would have  
the form $B_{c,d}(x_1,x_2)$, by Lemma~\ref{nonrestrictive}, 
and 
(ii)~it could not be a projection,
so we would have to have 
$\Sigma\vdash B_{c,d}(e,x)\approx x$ and 
$\Sigma\vdash B_{c,d}(x,f)\approx x$ for some
constants $e$ and $f$, and (iii)~$\{B_{c,d}(x,f)\}$ 
is a singleton class of the weak closure
of $\Sigma$, hence we do not have $\Sigma\vdash B_{c,d}(x,f)\approx x$
after all.

Finally, we sketch how to modify the proof 
of Theorem~\ref{nonrestrictive_thm} to eliminate
the restriction to finitely many constants in the
case where the algebras may be infinite.

Recall that we started with an algebra
$\m a$, enlarged it to $\m a_{z_1,\ldots,z_p,0}$
by iterating the one-point completion
construction, and then merged it with the model $\m v$
of $\Sigma$ to create $\m a_{\Sigma}$, which realized
$\Sigma$ and had the same growth rate as $\m a$.
In this construction, we used the one-point
completion construction $p$ times, where
$p$ was the number of equivalence classes of 
constant symbols under $\Sigma$-provable equivalence.
The only thing different here is that we may not have finitely
many equivalence classes of constant symbols. However, we may well-order
the equivalence classes of constants (say, by 
stipulating that $[c]<[d]$ if the least constant
in class $[c]$ is smaller than the least constant in $[d]$
under the well-order from the proof of Kelly's Theorem).
Now, rather than using the one-point completion 
construction $p$ times, we use the idea of the construction 
exactly once
to adjoin a well-ordered set $\{0\}\cup Z$ to $\m a$ to create
$\m a_{Z,0}$. Here the well-order is $0 < z_1 < z_2 < \cdots$,
with $0$ the least element, and $\lb Z; <\rb$ is a well-ordered
set for which there is a bijection $\varphi\colon [C]\to Z$
from the set of equivalence classes of constants.
The algebra has universe $A_{Z,0}$ equal to the disjoint
union of $A$, $Z$ and $0$. If $F$ is a function
symbol in the language of $\m a$, then it is defined
on $A_{Z,0}$ by 
\[
F^{\m A_{Z,0}}(\wec{a}) = \begin{cases}
F^{\m A}(\wec{a}) & \text{\rm if $\wec{a}\in A^n$};\\
\min\{\{a_1,\ldots,a_n\}\cap (\{0\}\cup Z)\} & \text{\rm else.}
\end{cases}
\]
We also define binary operations corresponding
to the operation $x\wedge y$ of the one-point
completion, namely $x\wedge_z y$ for $z\in Z\cup \{0\}$.
Here 
\[
x\wedge_z y = \begin{cases}
x & \text{\rm if $x=y$};\\
z & \text{\rm if $x\neq y$ and $x, y\in A\cup [z)$};\\
\min\{\{x,y\}\cap (\{0\}\cup Z)\} & \text{\rm else.}
\end{cases}
\]
Arguments similar to those in Theorem~\ref{partial}
show that a generating set for
$\m a^n$ also generates $\m a_{Z,0}^n$ and
any generating set for $\m a_{Z,0}^n$
contains a generating set for $\m a^n$.
We can merge $\m a_{Z,0}$
with a model $\m v$ of $\Sigma$
from Definition~\ref{V} to obtain a model
$\m a_{\Sigma}$, as we did for the proof
of Theorem~\ref{nonrestrictive_thm}.
Using the same arguments as before, it can be shown that
$\m a_{\Sigma}$ has the same growth rate as $\m a$
unless $\Sigma$ entails the existence of a pointed
cube term.
\end{remark}

\subsection{Pointed cube terms enforce 
polynomially bounded growth}\label{processing}

In the preceding subsection  
we proved that if $\Sigma$ is restrictive, then $\Sigma$
entails the existence of a pointed cube term.
We now prove the converse by showing that
if $\m a$ is an algebra with a pointed cube
term and sufficiently many of the small powers of $\m a$
are finitely generated, then all finite powers of $\m a$
are finitely generated and 
$d_{\m a}(n)$  is bounded above by a polynomial.

\begin{thm}\label{pointed_polynomial}
Let $\m a$ be an algebra with an
$m$-ary, $p$-pointed,
$k$-cube term, with at least one
constant symbol appearing in the cube identities
(so $p\geq 1$). If $\m a^{p+k-1}$ is finitely
generated, then all finite powers of $\m a$
are finitely generated and 
$d_{\m a}(n)$ is bounded above by 
a polynomial 
of degree at most 
$\log_w(m)$, where $w = 2k/(2k-1)$.
\end{thm}

The proof rests on the fact that a cube term, like
\begin{equation}\label{R}
\F\left(\begin{matrix}
1&x&2\\
x&2&3\\
\end{matrix}
\right) \approx
\left(\begin{matrix}
x\\
x\\
\end{matrix}
\right), 
\end{equation}
may be used to ``factor'' a typical tuple 
$\wec{a}\in A^n$ into simpler tuples:
\[
\F\left(
\left[
\begin{matrix}
{\mathbf 1}\\{\mathbf a}_2
\end{matrix}
\right],
\left[
\begin{matrix}
{\mathbf a}_1\\{\mathbf 2}
\end{matrix}
\right],
\left[
\begin{matrix}
{\mathbf 2}\\{\mathbf 3}
\end{matrix}
\right]
\right)=
\left[
\begin{matrix}
{\mathbf a}_1\\{\mathbf a}_2
\end{matrix}
\right].
\]
Here the $n$-tuple $\wec{a}$ has been split into
two blocks of coordinates 
of roughly equal size, $\wec{a} = \left[
\begin{matrix}
{\mathbf a}_1\\{\mathbf a}_2
\end{matrix}
\right]$,
then factored into 
$
\left[
\begin{matrix}
{\mathbf 1}\\{\mathbf a}_2
\end{matrix}\right],
\left[
\begin{matrix}
{\mathbf a}_1\\{\mathbf 2}
\end{matrix}\right],
\left[
\begin{matrix}
{\mathbf 2}\\{\mathbf 3}
\end{matrix}\right],
$
which are simpler than $\wec{a}$ in the sense that some of the
coordinates have been replaced by constants.
This factorization process can be iterated until 
the final factors 
have at most $k-1$ coordinate entries that have not
been replaced by elements from the set of constants.
The proof of the theorem 
develops such a factorization scheme under which
there are only polynomially many different types
of final factors,
and the collection of all final factors of a given type
lie in a subalgebra of $\m a^n$ isomorphic
to $\m a^{j}$ for some $j\leq p+k-1$. The set consisting of 
the generators of all of these
subalgebras is a polynomial-size generating set for $\m a^n$.

\begin{proof}
Suppose that the fact that $\F(x_1,\ldots,x_m)$
is a $p$-pointed $k$-cube term (with $p\geq 1$)
is witnessed by identities 
$\F(M)\approx [x,\ldots,x]^{\mathsf T}$,
where $M$ is a $k\times m$ matrix of variables and 
constant symbols, with at least one constant symbol,
where each column of $M$ contains a symbol
that is not $x$. 
Choose a constant symbol $c$ appearing in $M$, 
replace all instances of 
variables in $M$ that are not $x$ by $c$.
This produces another matrix $R$
with no variables other than $x$
which also witnesses that $\F$ is a $p$-pointed $k$-cube term.
The order of the 
$k$ rows identities, $\F(R)\approx [x,\ldots,x]^{\mathsf T}$,
is fixed once and for all.

We will refer to the function $\lambda\colon [m]\to [k]$
from the column indices to the row indices defined by the property
that $\lambda(j)=i$ exactly when $i$ is the least
index such that $R$ has a constant symbol
in its $i,j$-th position. Such $\lambda$ exists
because every column of $R$ contains at least
one constant symbol. (For the cube term in the example
immediately following the theorem statement
$\lambda\colon [3]\to [2]$ is the function
$\lambda(1)=\lambda(3) = 1, \lambda(2) = 2$.)

The factoring, or ``processing'', of tuples in $A^n$ will make use of 
an $m$-ary tree which we refer to as
the \emph{(processing) template}. We refer to nodes 
of the template by their
\emph{addresses}, which are finite strings in the alphabet
$[m] = \{1,\ldots,m\}$. The root node 
has empty address, and is denoted
$\wec{n}_{\emptyset}$. If $\wec{n}_{\sigma}$ is 
the node at address $\sigma$, then its children
are the nodes $\wec{n}_{\sigma 1},\ldots, \wec{n}_{\sigma m}$.

Each node $\wec{n}$ of the template is labeled 
by a subset $\ell(\wec{n})\subseteq [n]$.
(Recall that $n$ is the number 
appearing in the exponent of $A^n$.)
To define the labeling function $\ell$
we first specify a fixed
method for partitioning some subsets $U\subseteq [n]$.
Given a subset $U=\{u_1,\ldots,u_r\}\subseteq [n]$, 
consider it to be a linearly
ordered set $u_1 < \ldots < u_r$ under the order
inherited from $[n]$. Define
$\pi(U) = (U_1,\ldots,U_k)$ to be the ordered partition
of $U$ into $k$ consecutive nonempty intervals
that are as
equal in size as possible. 
In more detail, let 
\[
\pi(U)=(U_1,\ldots,U_k)=
(\{u_1,\ldots,u_{i_1}\},\{u_{i_1+1},\ldots,u_{i_2}\},\ldots,
\{u_{i_{k-1}+1},\ldots,u_{i_k}=u_r\}),
\]
where 
\[
u_1<\cdots<u_{i_1}<u_{i_1+1}<\cdots<u_{i_2}<u_{i_{k-1}+1}<\cdots <
u_{i_k}=u_r
\] 
(i.e., the cells of the partition are consecutive nonempty intervals)
and 
\[
|U_1|\geq \cdots \geq |U_k|\geq |U_1|-1
\]
(i.e., the cells are as equal sized as possible).
The $k$ appearing here as the number of cells of the partition
is the same $k$ as the one in the assumption that
$\F$ is a $k$-cube term.
In order for $\pi(U)$ to be defined, it is necessary
that $|U|\geq k$.

As mentioned earlier, the label on node $\wec{n}_{\sigma}$ 
will be 
some subset $\ell(\wec{n}_{\sigma})\subseteq [n]$.
Recursively define the labels as follows:
\begin{enumerate}
\item $\ell(\wec{n}_{\emptyset})=\emptyset$.
\item If all nodes between 
$\wec{n}_{\sigma}$ and $\wec{n}_{\emptyset}$ are labeled, 
$V$ is the union of labels
occurring between $\wec{n}_{\sigma}$
and the root $\wec{n}_{\emptyset}$, and
$\pi([n]\setminus V) = (U_1,\ldots,U_k)$,
then $\ell(\wec{n}_{\sigma i})= U_{\lambda(i)}$.
\end{enumerate}
In (2),
if $[n]\setminus V$ has fewer than $k$ elements, then 
it is impossible to partition it into $k$ nonempty intervals,
in which case there do not exist sufficiently
many labels for potential
children. In this case, we do not include
any descendants of $\wec{n}_{\sigma}$ in the template.

Let's illustrate our progress with the example 
started back at line (\ref{R}).
The following picture 
depicts the processing template in the case
$[n]=[5]=\{1,2,3,4,5\}$.
(Recall that $\lambda(1)=\lambda(3) = 1, \lambda(2) = 2$.)

\bigskip

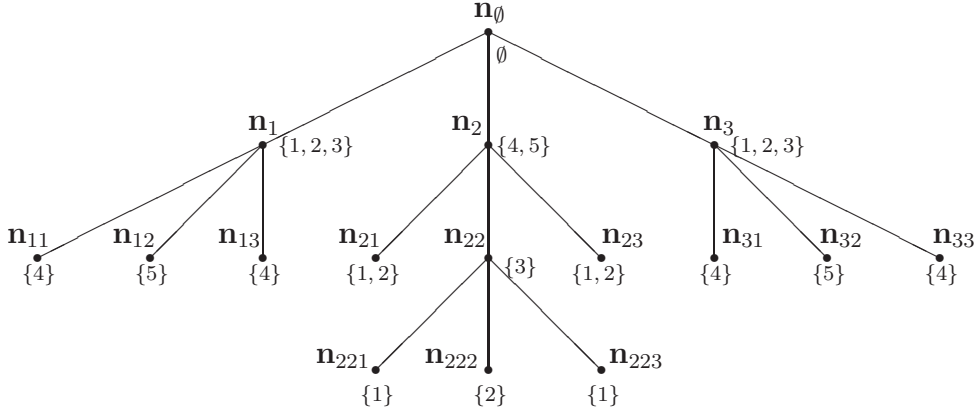
\begin{center}
\begin{picture}(300,150)
\setlength{\unitlength}{1mm}

\put(51,46){\tiny{$\emptyset$}}
\put(48,52){$\wec{n}_{\emptyset}$}

\put(22,34){\tiny{$\{1, 2, 3\}$}}
\put(18,37){$\wec{n}_{1}$}

\put(45,37){$\wec{n}_{2}$}
\put(51,34){\tiny{$\{4,5\}$}}

\put(78.5,37){$\wec{n}_{3}$}
\put(82,34){\tiny{$\{1, 2, 3\}$}}

\put(-14,22){$\wec{n}_{11}$}
\put(-12,17){\tiny{$\{4\}$}}

\put(0,22){$\wec{n}_{12}$}
\put(3,17){\tiny{$\{5\}$}}

\put(14,22){$\wec{n}_{13}$}
\put(18,17){\tiny{$\{4\}$}}

\put(30,22){$\wec{n}_{21}$}
\put(31,17){\tiny{$\{1, 2\}$}}

\put(44,22){$\wec{n}_{22}$}
\put(52,18){\tiny{$\{3\}$}}

%\put(56,22){$\wec{n}_{23}$}
\put(65,22){$\wec{n}_{23}$}
\put(61,17){\tiny{$\{1, 2\}$}}

%\put(74,22){$\wec{n}_{31}$}
\put(81,22){$\wec{n}_{31}$}
\put(78,17){\tiny{$\{4\}$}}

%\put(86,22){$\wec{n}_{32}$}
\put(94,22){$\wec{n}_{32}$}
\put(93,17){\tiny{$\{5\}$}}

%\put(99,22){$\wec{n}_{33}$}
\put(109,22){$\wec{n}_{33}$}
\put(108,17){\tiny{$\{4\}$}}

\put(27,6){$\wec{n}_{221}$}
\put(33,1){\tiny{$\{1\}$}}

\put(41.5,6){$\wec{n}_{222}$}
\put(48,1){\tiny{$\{2\}$}}

%\put(55,5){$\wec{n}_{223}$}
\put(66,6){$\wec{n}_{223}$}
\put(63,1){\tiny{$\{1\}$}}

\put(50,50){\line(-2,-1){30}}
\put(50,50){\line(0,-1){15}}
\put(50,50){\line(2,-1){30}}

\put(20,35){\line(-2,-1){30}}
\put(20,35){\line(-1,-1){15}}
\put(20,35){\line(0,-1){15}}

\put(50,35){\line(-1,-1){15}}
\put(50,35){\line(0,-3){15}}
\put(50,35){\line(1,-1){15}}

\put(80,35){\line(2,-1){30}}
\put(80,35){\line(1,-1){15}}
\put(80,35){\line(0,-1){15}}

\put(50,20){\line(-1,-1){15}}
\put(50,20){\line(0,-3){15}}
\put(50,20){\line(1,-1){15}}

\put(50,50){\circle*{1}}

\put(20,35){\circle*{1}}
\put(50,35){\circle*{1}}
\put(80,35){\circle*{1}}

\put(-10,20){\circle*{1}}
\put(5,20){\circle*{1}}
\put(20,20){\circle*{1}}

\put(35,20){\circle*{1}}
\put(50,20){\circle*{1}}
\put(65,20){\circle*{1}}

\put(80,20){\circle*{1}}
\put(95,20){\circle*{1}}
\put(110,20){\circle*{1}}

\put(35,5){\circle*{1}}
\put(50,5){\circle*{1}}
\put(65,5){\circle*{1}}

\end{picture}
\end{center}

Now we define precisely what is meant by processing.
Let $P=\{c_1,\ldots,c_p\}$ be the constant
symbols appearing in the cube identities
for $\F$.
A tuple $\wec{a}\in A^n$ is \emph{processed for node $\wec{n}_{\sigma}$} 
if there is a constant symbol $c\in P$
such that the $i$-th coordinate
of $\wec{a}$ is $c^{\m a}$ for all $i\in \ell(\wec{n}_{\sigma})$.
A tuple $\wec{a}$ is \emph{fully processed} if there is a 
path through the template from the root to a leaf
such that $\wec{a}$ is processed for each node in the path.

The processing template describes, in reverse order, 
a particular way to generate tuples in $\m a^n$.
Given a tuple $\wec{a}\in A^n$, we assign it
to the root $\wec{n}_{\emptyset}$ and denote it 
$\wec{a}_{\emptyset}$. This tuple 
$\wec{a} = \wec{a}_{\emptyset}$ is already processed for
$\wec{n}_{\emptyset}$, since this is an empty requirement.
Now, for each address $\sigma$ of a node in the template, 
we will construct $\wec{a}_{\sigma 1}, \ldots, \wec{a}_{\sigma m}$
from $\wec{a}_{\sigma}$ so that 
(i)~$\F^{\m a}(\wec{a}_{\sigma 1}, \ldots, \wec{a}_{\sigma m}) = \wec{a}_{\sigma}$,
and (ii)~each $\wec{a}_{\sigma i}$ is processed at
all nodes between $\wec{n}_{\sigma i}$ and $\wec{n}_{\emptyset}$.
Assign $\wec{a}_{\sigma i}$ to $\wec{n}_{\sigma i}$.
The original tuple $\wec{a}$ can be generated via
$\F^{\m a}$ by the fully processed tuples derived from
$\wec{a}$ in this way. The following claim is the heart of this argument.

\begin{clm}
Suppose that $\wec{n}_{\sigma}$ is an internal
node of the processing template.
Given an arbitrary tuple $\wec{a}\in A^n$, 
there exist tuples
$\wec{b}_{1}, \ldots, \wec{b}_{m}$ such that 
\begin{enumerate}
\item
$\F^{\m a}(\wec{b}_{1}, \ldots, \wec{b}_{m}) = \wec{a}$.
\item 
$\wec{b}_{i}$ is processed for node $\wec{n}_{\sigma i}$ for 
$i=1,\ldots,m$.
\item 
If $\wec{n}$ is a node between 
$\wec{n}_{\sigma}$ and 
$\wec{n}_{\emptyset}$, and $\wec{a}$ is processed
for $\wec{n}$, then each
$\wec{b}_{i}$ is also processed for $\wec{n}$
for $i=1,\ldots,m$.
\end{enumerate}
\end{clm}

\cproof
Let $V$ be the union of labels on nodes
between $\wec{n}_{\sigma}$ and $\wec{n}_{\emptyset}$.
If $\pi([n]\setminus V) = (U_1,\ldots,U_k)$,
then $\{V,U_1,\ldots,U_k\}$ is a partition of $[n]$
(with $V$ possibly empty).
For simplicity of expression, reorder coordinates
so that $\wec{a}$ and $\wec{b}_i$ can be written
$[\wec{a}_V, \wec{a}_{U_1}, \ldots, \wec{a}_{U_k}]^{\mathsf T}$
and
$[\wec{b}_{i,V}, \wec{b}_{i,U_1}, \ldots, \wec{b}_{i,U_k}]^{\mathsf T}$,
with coordinates from $V$ or $U_j$ grouped together.
Given $\wec{a}$, we need to solve for $\wec{b}_{i,V}$
and $\wec{b}_{i,U_j}$ in 
\begin{equation}\label{ddddddd}
\small{
\F^{\m a}(\wec{b}_1,\ldots,\wec{b}_m)=
\F^{\m a}\left( 
\left[
\begin{array}{c}
\wec{b}_{1,V}\\
\wec{b}_{1,U_1}\\
\vdots\\
\wec{b}_{1,U_k}\\
\end{array}
\right],
\ldots,
\left[
\begin{array}{c}
\wec{b}_{m,V}\\
\wec{b}_{m,U_1}\\
\vdots\\
\wec{b}_{m,U_k}\\
\end{array}
\right]
\right) = 
\left[
\begin{array}{c}
\wec{a}_{V}\\
\wec{a}_{U_1}\\
\vdots\\
\wec{a}_{U_k}\\
\end{array}
\right]
=\wec{a}
}
\end{equation}
in order to satisfy Item (1) of the claim.
We shall do so using the first cube identity in the
$V$-coordinates and the $U_1$-coordinates, and the 
$i$-th cube identity in the $U_i$-coordinates.

Whether $W=V$ or $W=U_i$, to solve 
$\F^{\m a}(\wec{b}_{1,W},\ldots,\wec{b}_{m,W})=\wec{a}_{W}$
for the $\wec{b}_{i,W}$'s using a particular cube identity,
take $\wec{b}_{i,W}=\wec{a}_W$ if there is an $x$ in the $i$-th
place of $F$ in the cube identity, and take 
$\wec{b}_{i,W}=[c^{\m a},\ldots,c^{\m a}]^{\mathsf T}$
if there is a $c$ in the $i$-th place
of the cube identity. It is not hard to see that this works, and so
(1) holds.

The label on node $\wec{n}_{\sigma i}$ is $U_{\lambda(i)}$.
The element $\lambda(i)\in [k]$ is the number of the first
cube identity that has some constant symbol
$c\in P$ in the $i$-th place of $F$.
Hence $\wec{b}_{i,U_{\lambda(i)}} = [c^{\m a},\ldots,c^{\m a}]^{\mathsf T}$.
Thus $\wec{b}_i$ is processed for node $\wec{n}_{\sigma i}$,
establishing (2).

If, in the first cube identity, there is an $x$ in the $i$-th place
of $F$, then $\wec{b}_{i,V} = \wec{a}_V$. If there is a constant
symbol $c\in P$ in the $i$-th place of $F$, then 
$\wec{b}_{i,V} = [c^{\m a},\ldots,c^{\m a}]^{\mathsf T}$.
In the latter case, $\wec{b}$ is processed at all 
coordinates in $V$, hence at all nodes between 
$\wec{n}_{\sigma}$ and $\wec{n}_{\emptyset}$. In the former
case, $\wec{b}_i$ is processed at any
node between 
$\wec{n}_{\sigma}$ and $\wec{n}_{\emptyset}$ where $\wec{a}$
is processed, since $\wec{b}_{i,V} = \wec{a}_V$. 
In either case, (3) holds.
The claim is proved.
\cqed

\medskip

The claim shows that we can attach any tuple $\wec{a}\in A^n$
to the root node and then 
process it down the tree using the cube identities
until we have
attached to the leaves the 
fully processed tuples associated to $\wec{a}$. 
Here we indicate the processing of a tuple
$\wec{a} \in A^5$ using the example template
given earlier.

\bigskip
\bigskip

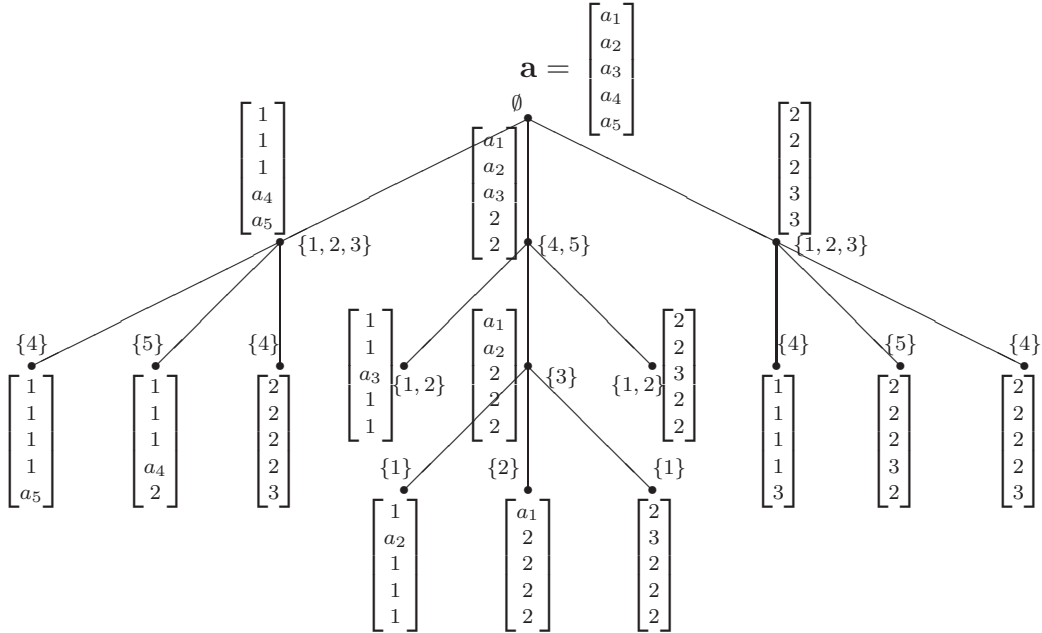
\begin{center}
\begin{picture}(300,180)
\setlength{\unitlength}{1.1mm}

\put(48,51){\tiny{$\emptyset$}}
%\put(48,52){$\wec{n}_{\emptyset}$}
\put(49,55){${\tiny{\wec{a} = \left[\begin{matrix} a_1\\a_2\\a_3\\a_4\\a_5\end{matrix}\right]}}$}

\put(22,34){\tiny{$\{1, 2, 3\}$}}
%\put(18,37){$\wec{n}_{1}$}
\put(14,43){${\tiny{\left[\begin{matrix} 1\\1\\1\\a_4\\a_5\end{matrix}\right]}}$}

%\put(45,37){$\wec{n}_{2}$}
\put(42,40){${\tiny{\left[\begin{matrix} a_1\\a_2\\a_3\\2\\2\end{matrix}\right]}}$}
\put(51,34){\tiny{$\{4,5\}$}}

%\put(78.5,37){$\wec{n}_{3}$}
\put(79,43){${\tiny{\left[\begin{matrix} 2\\2\\2\\3\\3\end{matrix}\right]}}$}
\put(82,34){\tiny{$\{1, 2, 3\}$}}

%\put(-14,22){$\wec{n}_{11}$}
\put(-14,10){${\tiny{\left[\begin{matrix} 1\\1\\1\\1\\a_5\end{matrix}\right]}}$}
\put(-12,22){\tiny{$\{4\}$}}

%\put(0,22){$\wec{n}_{12}$}
\put(1,10){${\tiny{\left[\begin{matrix} 1\\1\\1\\a_4\\2\end{matrix}\right]}}$}
\put(2,22){\tiny{$\{5\}$}}

%\put(14,22){$\wec{n}_{13}$}
\put(16,10){${\tiny{\left[\begin{matrix} 2\\2\\2\\2\\3\end{matrix}\right]}}$}
\put(16,22){\tiny{$\{4\}$}}

%\put(30,22){$\wec{n}_{21}$}
\put(27,18){${\tiny{\left[\begin{matrix} 1\\1\\a_3\\1\\1\end{matrix}\right]}}$}
\put(33.5,17){\tiny{$\{1, 2\}$}}

%\put(44,22){$\wec{n}_{22}$}
\put(42,18){${\tiny{\left[\begin{matrix} a_1\\a_2\\2\\2\\2\end{matrix}\right]}}$}
\put(52,18){\tiny{$\{3\}$}}

%\put(65,22){$\wec{n}_{23}$}
\put(65,18){${\tiny{\left[\begin{matrix} 2\\2\\3\\2\\2\end{matrix}\right]}}$}
\put(60,17){\tiny{$\{1, 2\}$}}

%\put(81,22){$\wec{n}_{31}$}
\put(77,10){${\tiny{\left[\begin{matrix} 1\\1\\1\\1\\3\end{matrix}\right]}}$}
\put(80,22){\tiny{$\{4\}$}}

%\put(94,22){$\wec{n}_{32}$}
\put(91,10){${\tiny{\left[\begin{matrix} 2\\2\\2\\3\\2\end{matrix}\right]}}$}
\put(93,22){\tiny{$\{5\}$}}

%\put(109,22){$\wec{n}_{33}$}
\put(106,10){${\tiny{\left[\begin{matrix} 2\\2\\2\\2\\3\end{matrix}\right]}}$}
\put(108,22){\tiny{$\{4\}$}}

%\put(27,6){$\wec{n}_{221}$}
\put(30,-5){${\tiny{\left[\begin{matrix} 1\\a_2\\1\\1\\1\end{matrix}\right]}}$}
\put(32,7){\tiny{$\{1\}$}}

%\put(41.5,6){$\wec{n}_{222}$}
\put(46,-5){${\tiny{\left[\begin{matrix} a_1\\2\\2\\2\\2\end{matrix}\right]}}$}
\put(45,7){\tiny{$\{2\}$}}

%\put(66,6){$\wec{n}_{223}$}
\put(62,-5){${\tiny{\left[\begin{matrix} 2\\3\\2\\2\\2\end{matrix}\right]}}$}
\put(65,7){\tiny{$\{1\}$}}

\put(50,50){\line(-2,-1){30}}
\put(50,50){\line(0,-1){15}}
\put(50,50){\line(2,-1){30}}

\put(20,35){\line(-2,-1){30}}
\put(20,35){\line(-1,-1){15}}
\put(20,35){\line(0,-1){15}}

\put(50,35){\line(-1,-1){15}}
\put(50,35){\line(0,-3){15}}
\put(50,35){\line(1,-1){15}}

\put(80,35){\line(2,-1){30}}
\put(80,35){\line(1,-1){15}}
\put(80,35){\line(0,-1){15}}

\put(50,20){\line(-1,-1){15}}
\put(50,20){\line(0,-3){15}}
\put(50,20){\line(1,-1){15}}

\put(50,50){\circle*{1}}

\put(20,35){\circle*{1}}
\put(50,35){\circle*{1}}
\put(80,35){\circle*{1}}

\put(-10,20){\circle*{1}}
\put(5,20){\circle*{1}}
\put(20,20){\circle*{1}}

\put(35,20){\circle*{1}}
\put(50,20){\circle*{1}}
\put(65,20){\circle*{1}}

\put(80,20){\circle*{1}}
\put(95,20){\circle*{1}}
\put(110,20){\circle*{1}}

\put(35,5){\circle*{1}}
\put(50,5){\circle*{1}}
\put(65,5){\circle*{1}}

\end{picture}
\end{center}

\bigskip
\bigskip
\bigskip
\bigskip
\bigskip

Each leaf of the template determines a \emph{type}
of fully processed tuples. Two fully processed
tuples $\wec{u}$ and $\wec{v}$
of the same type have the same processed
coordinates, and the same constant entries in the
processed coordinates. They differ only in the unprocessed
coordinates. For any given type there is a partition
of the $n$ coordinates into at most $p+k-1$ cells
where each unprocessed coordinate is a singleton cell
(there are at most $k-1$ of these cells) and 
all processed coordinates with a given constant
entry form a cell (there are at most $p$
of these cells). The collection of all tuples
of this type lie in the subalgebra of all tuples
constant on these cells, and this subalgebra is isomorphic to 
$\m a^{j}$ for some $j\leq p+k-1$.
The assumption of the theorem is that 
$\m a^{p+k-1}$ is finitely generated, say by $g$ elements.
This paragraph explains why $\m a^n$ has a subalgebra
generated by $\leq g$ elements 
(and isomorphic to $\m a^{j}$ for some $j\leq p+k-1$)
which contains all fully processed tuples of a given type.

For example, the fully processed tuple 
${\tiny{\left[\begin{matrix} 1\\1\\1\\a_4\\2\end{matrix}\right]}}$
from the preceding figure lies in the subalgebra of all
tuples of the form 
${\tiny{\left[\begin{matrix} x\\x\\x\\y\\z\end{matrix}\right]}}$,
which is isomorphic to $\m a^3$. Here $3\leq p+k-1=3+2-1=4$.

Now let's count the number of types.
Since the template is an $m$-ary tree, and the types
are determined by the leaves, the number of types is at 
most $m^r$ where $r$ is an upper bound on
the length of the longest branch in the processing template.
We must estimate $r$.

Let $V_0 = \ell(\wec{n}_{\emptyset}) = \emptyset$.
This represents the set of coordinate positions that have been
processed before the processing begins, i.e., no 
coordinate positions. As we progress down a 
branch in the template, $\wec{n}_{\emptyset}, 
\wec{n}_i, \wec{n}_{ij},\ldots, \wec{n}_{\sigma}$,
we may construct sets $V_{\sigma i} = V_{\sigma}\cup \ell(\wec{n}_{\sigma i})$,
where $V_{\sigma}$ represents the set of coordinate positions that have been
processed along this branch from $\wec{n}_{\emptyset}$ to $\wec{n}_{\sigma}$.
The unprocessed coordinate positions, $[n]\setminus V_{\sigma}$
are then divided evenly, 
$\pi([n]\setminus V_{\sigma}) = (U_1,\ldots,U_k)$, to appear
as labels of the children of $\wec{n}_{\sigma}$.
Thus, $|V_{\emptyset}|=0$ and
\begin{equation}\label{WWW}
|V_{\sigma i}| = |V_{\sigma}\cup \ell(\wec{n}_{\sigma i})|=
|V_{\sigma}|+|\ell(\wec{n}_{\sigma i})|.
\end{equation}
The useful parameter is the number 
$u_{\sigma}:=
|[n]\setminus V_{\sigma}|=n- |V_{\sigma}|$ 
of nodes that remain unprocessed 
after reaching $\wec{n}_{\sigma}$.
This parameter satisfies
$u_{\emptyset} = |[n]\setminus V_{\emptyset}|=n$ and, from (\ref{WWW}),
\begin{equation}\label{XXX}
u_{\sigma i} = (n - |V_{\sigma i}|) = 
(n-|V_{\sigma}|)-|\ell(\wec{n}_{\sigma i})| = 
u_{\sigma}-|\ell(\wec{n}_{\sigma i})|.
\end{equation}
Since 
$\pi([n]\setminus V_{\sigma}) = (U_1,\ldots,U_k)$ is an even division
of $[n]\setminus V_{\sigma}$ into $k$ sets, and  
$\ell(\wec{n}_{\sigma i})=U_{\lambda(i)}$, we get 
\begin{equation}\label{YYY}
|\ell(\wec{n}_{\sigma i})|=|U_{\lambda(i)}|\geq 
\lfloor(n-|V_{\sigma}|)/k\rfloor = 
\lfloor u_{\sigma}/k\rfloor.
\end{equation}
Combining (\ref{XXX}) and (\ref{YYY}) we have
\[
u_{\sigma i} 
\leq 
u_{\sigma} 
- \lfloor u_{\sigma}/k\rfloor = 
\left\lceil\left(\frac{k-1}{k}\right)u_{\sigma}\right\rceil.
\]
In order to avoid considering truncation error, we
use the following fact, whose proof we leave to the reader.
\begin{clm}
If $u\geq k\geq 1$, then 
$\left\lceil \left(\frac{k-1}{k}\right)u\right\rceil\leq
\left(\frac{2k-1}{2k}\right)u
$.\cqed
\end{clm}
Hence
\[
u_{\sigma i} 
\leq 
\left(\frac{2k-1}{2k}\right)u_{\sigma}
\]
for each $\sigma$, and therefore 
\[
u_{\sigma} 
\leq 
\left(\frac{2k-1}{2k}\right)^{|\sigma|}u_{\emptyset} =
\left(\frac{2k-1}{2k}\right)^{|\sigma|}n
\]
for each $\sigma$.
If, for some $r$, it happens that 
$\left(\frac{2k-1}{2k}\right)^{r}n < k$,
then there are fewer than $k$ unprocessed nodes
at address $\sigma$ for any $\sigma$
satisfying $|\sigma|\geq r$. Such an 
$r$ is an upper bound on the length of 
paths through the template.

Solving $\left(\frac{2k-1}{2k}\right)^{r}n < k$
for $r$ we obtain that any $r > \log_w(n/k)$, 
$w = \frac{2k}{2k-1}$, is an upper bound
on the length of paths in the template;
hence $r=\log_w(n/k)+1$ is such a bound.
Hence the number of types of fully processed tuples
is no more than
\[
m^r=
m^{\log_w(n/k)+1} =
m^{\log_w(n/k)}m
= (n/k)^{\log_w(m)}m
\in O(n^{\log_w(m)}).
\]

Recall that for each type, the set of fully processed
tuples lies in a $g$-generated subalgebra of $\m a^n$.
Collecting these generators yields a set of size
$O(n^{\log_w(m)})$ which generates all fully processed
tuples, hence generates $\m a^n$.
\end{proof}

This theorem deals only with the case $p\geq 1$.
We describe next how to refine the estimate
in the case $p=1$ and how to derive the result
for $p=0$ from the $p=1$ case.

\begin{cor}\label{pointed_polynomial_cor}
If $\m a^k$ is a finitely generated algebra with a $0$-pointed
or $1$-pointed $k$-cube term, then $d_{\m a}(n)\in O(n^{k-1})$. 
\end{cor}

\begin{proof}
Suppose that $\m a$ has a $1$-pointed $k$-cube term,
and that $c$ is the one constant that appears
among the cube identities. Then a fully processed tuple
$\wec{a}$ has $c$ in every processed coordinate position,
and has at most $k-1$ unprocessed coordinate positions.
Hence the set of tuples with a 
$c$ in all but at most $k-1$ positions 
contains all the fully processed tuples, and therefore
is a generating set for $\m a^n$. 

Suppose that $\m a^k$ is $g$-generated.
If $U\subseteq [n]$ has size $k-1$, then the subalgebra 
$\m a[U]$ of tuples in $\m a^n$ that are constant off of $U$ is isomorphic
to $\m a^k$, and so is also $g$-generated. This subalgebra
contains all tuples that have entry $c$ off of $U$.
If we collect the $g$ generators for $\m a[U]$ 
for each $k-1$ element subset $U\subseteq [n]$ we obtain a 
set of size
$
\binom{n}{k-1}g
$
which
generates $\m a^n$. Therefore $d_{\m a}(n)\leq \binom{n}{k-1}g\in O(n)$.

Now suppose that $F(x_1,\ldots,x_m)$ is a $0$-pointed $k$-cube term
of $\m a$ and that the cube identities are
\begin{equation}\label{F}
\F(M)\approx \left(\begin{array}{c} x\\ \vdots\\x\end{array}\right).
\end{equation}
Expand $\m a$ to an algebra $\m b$ by adjoining
a single constant, say $c$. 
Replace all variables
other than $x$ in (\ref{F}) with $c$ to obtain 
identities witnessing that $F(x_1,\ldots,x_m)$
is a $1$-pointed $k$-cube term for $\m b$.
Hence $d_{\m b}(n)\in O(n^{k-1})$ by the earlier
part of the argument. Now 
$d_{\m a}(n)\in O(n^{k-1})$ by 
Theorem~\ref{basic_estimates}~(4).
\end{proof}

In 
\cite{paper2}
we improve this result by showing
that finite algebras with a 
$0$-pointed $k$-cube term have logarithmic
or linear growth.

Let's combine the results of this subsection 
with the results of the previous subsection.

\begin{thm}\label{summary}
The following are equivalent for a set 
$\Sigma$ of basic identities in which only finitely
many constant symbols occur. 
\begin{enumerate}
\item[(1)] $\Sigma$ is restrictive.
[That is, 
the class of $d$-functions of algebras
is not equal to the class of $d$-functions of 
algebras that realize $\Sigma$.]
\item[(2)] $\Sigma$ is restrictive for finite algebras.
[The class of $d$-functions of finite algebras
is not equal to the class of $d$-functions of finite 
algebras that realize $\Sigma$.]
\item[(3)] The variety axiomatized by $\Sigma$
has a pointed cube term.
\item[(4)] The variety axiomatized by $\Sigma$
has a pointed cube term
of the form $F(x_1,\ldots,x_m)$, 
where $m\geq 2$, $F$ is a function
symbol occurring in $\Sigma$, and 
the variables $x_1,\ldots,x_m$ are distinct.
\item[(5)] If $\m a$ is an algebra realizing $\Sigma$
and $d_{\m a}(n)$ is finite for all $n$,
then $d_{\m a}(n)$ is bounded above by a polynomial.
\item[(6)] There is no (finite) algebra $\m a$ realizing
$\Sigma$ such that $d_{\m a}(n) = 2^n$ for all $n$.
\end{enumerate}
\end{thm}

\begin{proof}

[$(1)\Rightarrow(3)$ and $(2)\Rightarrow(3)$] 
Theorem~\ref{nonrestrictive_thm}.

[$(3)\Leftrightarrow(4)$] Lemma~\ref{nonrestrictive}.

[$(3)\Rightarrow(5)$] Theorem~\ref{pointed_polynomial}
and Corollary~\ref{pointed_polynomial_cor}.

[$(5)\Rightarrow(6)$] $d_{\m a}(n) = 2^n$ is not bounded
above by a polynomial.

[$(6)\Rightarrow(1)$ and $(6)\Rightarrow(2)$] 
If (1) or (2) fails then $\Sigma$ is
nonrestrictive (for finite algebras).
Thus, there exists a (finite) algebra 
$\m A$ realizing $\Sigma$ with the same growth
rate $d_{\m A}(n)=2^n$ as the 
$2$-element set equipped with no operations. Hence (6) fails.
\end{proof}

\subsection{Finite algebras with polynomial growth}\label{poly_growth}

In this subsection we prove that the
bound on growth
rates for finite algebras with $1$-pointed
$k$-cube terms, established in
Corollary~\ref{pointed_polynomial_cor},
is sharp. 

\begin{thm}\label{example}
For each $k\geq 2$ there is a finite algebra 
with a $1$-pointed $k$-cube term whose growth rate
satisfies $d_{\m a}(n)\in \Theta(n^{k-1})$.
\end{thm}

\begin{proof}
We shall first construct a partial algebra 
with the desired growth rate, 
then modify it slightly to obtain a total algebra
satisfying the hypotheses of the theorem.

The universe of the partial
algebra will be $A = \{a_1,\ldots,a_q,1\}$.
We equip this set with a partial
$k$-ary operation $F$ which satisfies
\[
F^{\m a}(1,x,\ldots,x,x) =
F^{\m a}(x,1,\ldots,x,x) = 
\cdots =
F^{\m a}(x,x,\ldots,x,1) = x
\]
for each $x\in A$, and which is undefined otherwise.
Thus, $F^{\m a}$ is a partial near unanimity operation
that is defined only on the nearly unanimous tuples
where the lone dissenter is $1$ and on the tuple
whose entries are unanimously $1$.
Set $\m a = \langle A; F\rangle$. 

We shall prove the exact formula
\begin{equation}\label{d_formula}
d_{\m a}(n)=
\binom{n}{0} + q\binom{n}{1} + q^2\binom{n}{2} + \cdots + q^{k-1}\binom{n}{k-1}
\end{equation}
for this partial algebra,
which is a polynomial in $n$ of degree $k-1$, since 
$k = {\rm arity}(F)$ and 
$q=|A|-1$ are fixed. 
This will show that $\m a$ is a $(q+1)$-element
partial algebra with $d_{\m a}(n)\in\Theta(n^{k-1})$.

Choose and fix $n$. Define
the \emph{support} of a tuple $\wec{a}\in A^n$ to be the 
subset $\supp(\wec{a})\subseteq [n]$ 
consisting of indices $s$ where $a_s\neq 1$.
The proof involves showing that the set of all tuples
whose support has size at most $k-1$
is the unique minimal generating set for $\m a^n$.
To set up language for the argument, call a tuple 
$\wec{b}\in A^n$ an \emph{essential generator}
if it is contained in any generating set for $\m a^n$.

\begin{clm}\label{support}
If $S\subseteq [n]$ and $G\subseteq A^n$, then let $G_S$ denote the 
set of tuples
in $G$ that have support contained in $S$.
If $\wec{a}\in \lb G\rb$ has support in $S$, 
then $\wec{a}\in \lb G_S\rb$.
\end{clm}

\cproof
In $\m a$, we have 
\[
F^{\m a}(x_1,\ldots,x_k) = 1\Longleftrightarrow x_1 = x_2 = \cdots = x_k = 1.
\]
Hence, in $\m a^n$, 
if $F^{\m a^n}(\wec{g}_1,\ldots,\wec{g}_k)$
is defined and equal to $\wec{b}$,
then $i\notin \supp(\wec{b})$ if and only if $i\notin \supp(\wec{g}_i)$
for any $\wec{g}_i$. Equivalently, 
\begin{equation}\label{generate}
\supp(F^{\m a^n}(\wec{g}_1,\ldots,\wec{g}_k)) = \bigcup_{i=1}^k \supp(\wec{g}_i)
\end{equation}
whenever $F^{\m a^n}(\wec{g}_1,\ldots,\wec{g}_k)$ is defined.
Now let $G(0) = G$, $G_S(0)=G_S$,
$G(j+1) = G(j)\cup F^{\m a^n}(G(j),\ldots,G(j))$, and
$G_S(j+1) = G_S(j)\cup F^{\m a^n}(G_S(j),\ldots,G_S(j))$.
By induction on $j$, using (\ref{generate}),
it can be shown that any tuple in $G(j)$ that has support
in $S$ lies in $G_S(j)$. Since $\lb G\rb = \bigcup_j G(j)$
and $\lb G_S\rb = \bigcup_j G_S(j)$, 
any tuple in $\lb G\rb$ with support in $S$ lies in $\lb G_S\rb$.
\cqed

\medskip

\begin{clm}\label{empty}
The tuple $\hat{1}=[1,1,\ldots,1]^{\mathsf T}$ of empty support
is an essential generator.
\end{clm}

\cproof
This follows immediately from Claim~\ref{support}.
\cqed

\medskip

\begin{clm}\label{supp}
Any tuple whose support has size at most $k-1$
is an essential generator of $\m a^n$.
\end{clm}

\cproof
Let $\wec{b}\in A^n$ be a tuple of support $S$
where $1\leq |S|\leq k-1$. Without loss of generality,
$S = [\ell]  =\{1,\ldots,\ell\}$ for some $1\leq \ell\leq k-1$. 
In order to obtain a contradiction to the claim,
assume that $\wec{b}$ is not an essential generator.
Then $\wec{b}$ can be generated by elements different from
$\wec{b}$, so
the equation $F^{\m a^n}(\wec{x}_1,\ldots,\wec{x}_k) = \wec{b}$
can be solved for the $\wec{x}_i$ in such a way
that $\wec{b}\notin \{\wec{x}_1,\ldots,\wec{x}_k\}$.
Moreover, by (\ref{generate}), the $\wec{x}_i$'s must be taken
from the tuples whose support is contained in $S$.
The equation to be solved is therefore:
\begin{equation}\label{dividers}
\small{
F^{\m a^n}(\wec{x}_1,\ldots,\wec{x}_k)=
F^{\m a^n}\left( 
\left[
\begin{array}{c}
x_{1,1}\\
\vdots\\
\underline{\;\;x_{\ell,1}\;\;}\\
1\\
\vdots\\
1\\
\end{array}
\right],
\ldots,
\left[
\begin{array}{c}
x_{1,k}\\
\vdots\\
\underline{\;\;x_{\ell,k}\;\;}\\
1\\
\vdots\\
1\\
\end{array}
\right]
\right) = 
\left[
\begin{array}{c}
b_1\\
\vdots\\
\underline{\;\;\;b_{\ell}\;\;\;}\\
1\\
\vdots\\
1\\
\end{array}
\right]=\wec{b}.
}
\end{equation}
We have introduced horizontal segments
as dividers separating the coordinates
in $S = [\ell]$ from the
remaining coordinates in order
to make the argument clearer.
Since $F^{\m a^n}(\wec{x}_1,\ldots,\wec{x}_k)$ is defined,
every row above the dividers
is a nearly unanimous row with 
exactly one $1$. Hence there are exactly
$\ell$ $1$'s above the dividers.
This means that there are at most $\ell$ columns which
contain a $1$ above the dividers. Since there are
$k$ such columns, and $k>\ell$, there is a column
$\wec{x}_j$ that contains no $1$ above the dividers.
Since the $i$-th row above the dividers is nearly
unanimous with majority value $b_i$, the column $\wec{x}_j$
which contains no $1$'s above the dividers
is exactly $\wec{b}$. This contradicts the assumption that
$\wec{b}\notin \{\wec{x}_1,\ldots,\wec{x}_k\}$,
showing that $\wec{b}$ is indeed an essential generator.
\cqed

\medskip

\begin{clm}
$\m a^n$ is generated by the tuples whose support 
has size at most $k-1$.
\end{clm}

\cproof
It is enough to show that if $\wec{b}$ has support
$S$ of size $\ell \geq k$, then $\wec{b}$ can be generated
from tuples whose support is properly contained in $S$.
It is enough to prove this in the case where $S = [\ell]$.
For this we must explain how to solve
\begin{equation}\label{dividers2}
\small{
F^{\m a^n}(\wec{x}_1,\ldots,\wec{x}_k)=
F^{\m a^n}\left( 
\left[
\begin{array}{c}
x_{1,1}\\
\vdots\\
\underline{\;\;x_{\ell,1}\;\;}\\
1\\
\vdots\\
1\\
\end{array}
\right],
\ldots,
\left[
\begin{array}{c}
x_{1,k}\\
\vdots\\
\underline{\;\;x_{\ell,k}\;\;}\\
1\\
\vdots\\
1\\
\end{array}
\right]
\right) = 
\left[
\begin{array}{c}
b_1\\
\vdots\\
\underline{\;\;\;b_{\ell}\;\;\;}\\
1\\
\vdots\\
1\\
\end{array}
\right]=\wec{b}
}
\end{equation}
when $\ell \geq k$ in such a way that every 
column contains at least one $1$ above the dividers
and the $i$-th row above the dividers is nearly unanimously
equal to $b_i$. This is easy to do.
Set $x_{1,1} = \cdots = x_{k,k} = 1$, then 
put exactly one $1$ arbitrarily in each 
of rows $k+1$ to $\ell$, then fill
in the remaining entries above the dividers
so that the $i$-th row above the dividers is nearly unanimously
equal to $b_i$. 
\cqed

\medskip

We have established up to this point that the set of tuples
of support of size at most $k-1$ is the unique minimal
generating set for $\m a^n$.
To complete the proof that the partial algebra $\m a$ has the specified
growth rate, observe that the number of tuples with 
support $S$ is $(|A|-1)^{|S|} = q^{|S|}$, so the number of tuples
whose support has size $i$ is $q^i\binom{n}{i}$.
This yields the formula $d_{\m a}(n)=\sum_{i=0}^{k-1} q^i\binom{n}{i}$. 

The one-point completion, $\m a_0$, is a total algebra
with the same growth rate as $\m a$. Let $\m b$
be the expansion of $\m a_0$ by one constant symbol
$1$ whose interpretation is $1^{\m b} = 1$.
The operation $F^{\m b}$ still
satisfies
\[
F^{\m b}(1,x,\ldots,x,x) =
F^{\m b}(x,1,\ldots,x,x) = 
\cdots =
 F^{\m b}(x,x,\ldots,x,1) = x
\]
for each $x\in A_0$, so it is a $1$-pointed
$k$-cube term for $\m b$.

By Theorem~\ref{partial}
$\m a^n$ and $\m a_0^n$ have the same unique minimal generating
set, $G$, which is the set of all tuples with support at most $k-1$; 
this set contains $\hat{1}$.
The algebra $\m b$ must also have a unique minimal generating
set, namely the set obtained from $G$
by deleting $\hat{1}=1^{\m b^n}$.
Thus $d_{\m b}(n) = d_{\m a}(n)-1 = \sum_{i=1}^{k-1} q^i\binom{n}{i}\in O(n^{k-1})$. 
\end{proof}

\subsection{Pointed cube polynomials can be avoided}\label{avoid}

We have established that if $\m a$ is an algebra whose
$d$-function assumes only finite values, and
$\m a$ has a pointed cube term
(or pointed cube polynomial for that matter),
then $d_{\m a}(n)$ is bounded above by a polynomial function of $n$.
The same growth rate can be obtained without 
a pointed cube term (or polynomial), as we show next.

\begin{thm}\label{avoid_thm}
Let $\m a$ be an algebra with $|A|>1$ whose
$d$-function assumes only finite values.
There
is an algebra $\m b$ such that $d_{\m b}(n) = d_{\m a}(n)$ for all $n$, and
\begin{enumerate}
\item[(1)] 
the universe of $\m b$ is $B:=A\cup\{0,z\}$ where
$0\neq z$ and $0, z\notin A$, 
\item[(2)] $\m b$ has a meet semilattice term operation, $\wedge$, 
with respect to which
$\m b$ has height one and least element $0$, and
\item[(3)] if $p(x,\wec{y})$ is an $m$-ary polynomial of 
$\m b$ in which $x$ actually appears and
$p(z,\wec{b}) = z$ for some $\wec{b}\in B^{m-1}$,
then either $p(x,\wec{y}) \approx x$
or else
$p(x,\wec{y}) \approx x\wedge q(\wec{y})$ for some
polynomial $q$ in which $x$ does not appear.
\end{enumerate}
In particular, $\m b$ does not have a pointed cube polynomial.
\end{thm}

\begin{proof}
Let $G(n):=\{\wec{g}_{n,1},\ldots,\wec{g}_{n,d(n)}\}$
be a least size generating set for $\m a^n$.
Let $\m a_z$ be the one-point
completion of $\m a$ with the element $z\;(\notin A)$ taken to be the
new point added. According to Theorem~\ref{partial}, the set
$G(n)$ is also a least size generating set for $\m a_z$.
Next, copying the idea of the construction
in Theorem~\ref{big}, for each $\wec{a}\in (A_z)^n$ 
introduce a partial 
operation $F_{\wec{a}}(x_1,\ldots,x_{d(n)})$ 
on $A_z$ 
with the properties that
(i)~the vector equation 
\begin{equation}\label{definition_Fb}
F_{\wec{a}}(\wec{g}_{n,1},\ldots,\wec{g}_{n,d(n)}) = \wec{a}
\end{equation}
holds coordinatewise, 
and (ii)~$F_{\wec{a}}$ is defined only on those tuples
required to make this equation hold.
Let $\overline{\m A}_z$ be the set $A_z$ 
equipped with these partial operations.
The partial algebra $\overline{\m A}_z$ is a reduct of 
$\m a_z$, so in passing from $\m a_z$ to 
$\overline{\m a}_z$ we may have lost
but not gained some generating subsets of powers.
On the other hand,
our choice of the partial operations guarantees that
$G(n)$ still generates the $n$-th power of
$\overline{\m A}_z$. This implies 
that $G(n)$ is a least size generating set for 
$\overline{\m A}_z^n$ for each $n$, and hence that 
the $d$-functions of $\overline{\m A}_z$ and
$\m a_z$ are the same.
Finally, let $\m b=(\overline{\m a}_{z})_{0}$ be the one-point completion of 
$\overline{\m A}_z$ with 
the element $0\;(\notin A\cup\{z\})$ taken to be the
new point added. 
With this choice the universe of $\m B$ is $B=A\cup\{0,z\}$.
Again citing Theorem~\ref{partial}, we see
that $G(n)$ is a least size generating set for
$\m b^n$.

At this point we have that $d_{\m b}(n) = d_{\m a}(n)$ for
all $n$, and also, by construction, that Items (1) and (2) hold.
(Here the meet operation referred to in Item (3) is the 
one introduced in the second one-point completion, the one
used to construct $\m b$ from $\overline{\m a}_z$.)

Let's prove that Item (3) holds.
Our argument depends on a Key Fact:
$z$ does not appear in any coordinate of any
tuple in $G(n)$ for any $n$, hence 
$z$ does not appear in any tuple
in the domain of any partial
operation of the form $F_{\wec{a}}$.
This implies that any basic operation of $\m b$ of the form
$(F_{\wec{a}})_0$ (Definition~\ref{one-point}) 
assigns the value $0$ to any tuple containing a $z$ (or a $0$).

We first prove that if $p(x,\wec{y})$ is an $m$-ary polynomial of $\m b$
in which $x$ appears and $\wec{b}\in B^{m-1}$, then 
$p(z,\wec{b})\in \{0,z\}$.  Arguing by induction
on the complexity of $p$, we need to consider the cases 
were $p$ is a constant, a variable, or of the form 
\begin{equation}\label{inductive_arg}
p(x,\wec{y}) 
= F(p_1(x,\wec{y}) ,\ldots,p_{\ell}(x,\wec{y}))
\end{equation}
where $F = (F_{\wec{a}})_0$ or $F = \wedge$.
The polynomial $p$ cannot be a constant, since $x$ appears in $p$.
If $p$ is a variable, it must be $x$, since $x$ appears in $p$.
In this case $p(z,\wec{b})=z\in \{0,z\}$, as claimed.
If (\ref{inductive_arg}) holds in the case where 
$F = (F_{\wec{a}})_0$, then by induction we have
$p_i(z,\wec{b})\in \{0,z\}$ for at least one $i$, hence by 
the Key Fact we obtain that
\[
p(z,\wec{b}) 
= (F_{\wec{a}})_0(p_1(z,\wec{b}) ,\ldots,p_{\ell}(z,\wec{b})) = 0\in\{0,z\},
\]
as claimed.
If (\ref{inductive_arg}) holds in the case where 
$F = \wedge$, then 
by induction we have
$p_i(z,\wec{b})\in \{0,z\}$ for at least one $i$,
hence $p_i(z,\wec{b})\leq z$. 
It follows that $p(z,\wec{b}) = p_1(z,\wec{b})\wedge p_2(z,\wec{b})\leq z$,
so, since 
$\langle B;\wedge\rangle$ has height one, we get that 
$p(z,\wec{b}) \in \{0,z\}$.

Now we prove Item (3) by induction on the complexity of $p$.
Under the assumptions of Item (3) 
the polynomial $p$ cannot be a constant, since $x$ appears in $p$.
If $p$ is a variable, it must be $x$, since $x$ appears in $p$,
in which case $p(x,\wec{y}) = x$ for all 
$x$ and $\wec{y}$, and Item (3) holds.
Now assume that (\ref{inductive_arg}) holds in the case where 
$F = (F_{\wec{a}})_0$, and fix a tuple $\wec{b}\in B^{m-1}$  
satisfying $p(z,\wec{b})=z$ (the existence of such a $\wec{b}$
is assumed in Item~(3)).
Since $x$ appears in $p$, 
by the induction hypothesis we have
$p_i(x,\wec{y}) = x$ or $x\wedge q_i(\wec{y})$ for some 
$i$ and some polynomial $q_i$.
In either case, $p_i(z,\wec{b})\in \{0,z\}$ by the result
of the preceding paragraph, and this gives us the right hand
equality (the only nontrivial equality) in:
\[
z = p(z,\wec{b}) = 
(F_{\wec{a}})_0(p_1(z,\wec{b}) ,\ldots,p_{\ell}(z,\wec{b})) = 0.
\]
This is a contradiction, which shows that this case cannot occur.
Finally, if 
$
p(x,\wec{y}) = p_1(x,\wec{y}) \wedge p_2(x,\wec{y})
$
and $\wec{b}\in B^{m-1}$ is such that
$p(z,\wec{b}) = z$, then $p_i(z,\wec{b})=z$ for $i=1,2$,
since $z$ is meet irreducible in $\langle B; \wedge\rangle$.
If $x$ appears in both $p_1(x,\wec{y})$ and 
$p_2(x,\wec{y})$, then by induction
both have the form $x$ or $x\wedge q_i(\wec{y})$.
Hence $p(x,\wec{y})$ has the form 
\[
x\wedge x,\quad x\wedge (x\wedge q_2(\wec{y})),\quad
(x\wedge q_1(\wec{y}))\wedge x,\;\quad\textrm{or}\quad\;
(x\wedge q_1(\wec{y}))\wedge (x\wedge q_2(\wec{y})),
\]
each of which has the form $x$ or $x\wedge q(\wec{y})$ for
some polynomial $q$. A similar conclusion is reached
if $x$ appears in one of the polynomials $p_i(x,\wec{y})$ 
but not the other. Hence Item (3) holds.

To complete the proof of the theorem
we argue that $\m b$ does not have a pointed cube
polynomial. By way of contradiction, assume that $p(x_1,\ldots,x_m)$
is such a polynomial and that $M$ is a $k\times m$
matrix of variables and constants such that
$p(M)\approx [x,\ldots,x]^{\mathsf T}$ and every 
column of $M$ contains at least one entry that is not $x$.
In fact, as we have seen before, by substituting 
constants for the variables different from $x$ we may assume that
the entries of $M$ are constants or $x$ and that each column
contains at least one constant. We may also assume that $p$
depends on all of its variables,
hence that each of $x_1,\ldots, x_m$ appears in $p$.

Here are some elementary consequences of our assumptions.
\begin{enumerate}
\item[(a)]
Each row of $M$ must contain at least one $x$, since otherwise
we may derive from the associated cube identity that $x\approx y$
holds in $\m b$. By permuting columns of $M$
(hence reordering the variables of $p$), we
assume that the first entry of the first row is $x$.
\item[(b)] 
The first column of $M$ contains a constant, which cannot
be in the first row. By permuting the later rows of $M$
(hence reordering the cube identities), we assume
that the first entry of the second row of $M$ is a constant.
There is an $x$ somewhere on the second row, by (a),
and permuting the later columns we may assume that it is in the second
position of the second row.
\end{enumerate}
These consequences mean that the first two cube identities
look like $p(x,b_2,\wec{b})\approx x$ and 
$p(c_1,x,\wec{c})\approx x$ where all $b_i, c_j\in B\cup \{x\}$
and $c_1$ is constant.
If we substitute $z$ for each $x$ in these equations we get
$p(z,b_2',\wec{b}') = z$ and 
$p(c_1',z,\wec{c}') = z$, where the primes on
elements and tuples indicate that 
the $x$'s in the string have been replaced 
by $z$'s and constants remain the same. Applying Item (3) 
of this theorem to these equalities 
we obtain that 
\[
p(x_1,x_2,\wec{y}) = x_1\wedge q_1(x_2,\wec{y}) = x_2\wedge q_2(x_1,\wec{y}),
\]
where $x_i$ does not appear in $q_i$. By meeting $p$ with itself
we obtain that
\[
p(x_1,x_2,\wec{y}) = 
(x_1\wedge q_1(x_2,\wec{y}))\wedge(x_2\wedge q_2(x_1,\wec{y})).
\]
Now the second
cube identity may be written 
\[
x=p(c_1,x,\wec{c})=
(c_1\wedge q_1(x,\wec{c}))\wedge(x\wedge q_2(c_1,\wec{c}))\leq c_1.
\]
This implies $x\leq c_1$ for all $x\in B$, and
therefore that the element $c_1\in B$ is the largest
element of $\langle B;\wedge\rangle$. But this semilattice
has no largest element, since it has at least 4 elements
and has height 1.
This contradiction proves that $\m b$ has no 
pointed cube polynomial.
\end{proof}

\subsection{Exponential growth}\label{exp}

If $\m a$ has exponential growth and $\m b$ has arbitrary growth,
then $\m a\times \m b$ has exponential growth
according to Theorem~\ref{basic_estimates}~(2). Hence
it is probably unrealistic to expect any 
meaningful classification of algebras with exponential
growth. This subsection will therefore be limited to 
identifying one property that forces exponential
growth. We will use the property to show that
the variety generated by the 2-element implication algebra,
$\lb \{0,1\}; \to \rb$, contains a chain of finite
algebras $\m a_1\leq \m a_2\leq \cdots$, each one
a subalgebra of the next, where 
$\m a_i$ has logarithmic growth when $i$ is odd
and exponential growth when $i$ is even.

We explore a simple idea:
Suppose that $\m a$ is finite and 
$u$ and $v$ are distinct elements of $A$.
If every element of $\{u, v\}^n$ is an essential
generator of $\m a^n$ for each $n$, then the growth rate
of $\m a$ must be at least $2^n$. A
way to force some tuple $\wec{t}\in \{u, v\}^n$
to be an essential generator of $\m a^n$ is to 
arrange that $A^n\setminus\{\wec{t}\}$
is a subuniverse of $\m a^n$. This 
can be accomplished by imposing 
an irreducibility condition on each coordinate $t$ of $\wec{t}$,
or equivalently by requiring that the complementary
set $A\setminus\{t\}$ behaves like an ideal.
For this to work it is enough that 
$A\setminus\{t\}$ behaves like a 1-sided semigroup-theoretic
ideal, so we introduce a definition that captures
this notion for an arbitrary algebraic signature.

\begin{df}
Let $\sigma=(F,\alpha)$ be an algebraic signature.
I.e., let $F$ be a set (of operation symbols)
and let $\alpha\colon F\to \omega$ be a function
(assigning arity). Let $F_0\subseteq F$ be the set
consisting of those $f\in F$ such that $\alpha(f)>0$.
($F_0$ is the set of nonnullary symbols.)
A \emph{selector}
for $\sigma$ is a function $\phi\colon F_0\to \omega$ 
such that $1\leq \phi(f)\leq \alpha(f)$ for each $f\in F_0$.
($\phi$ selects one of the places of the function
symbol $f$.)

If $\phi$ is a selector for $\sigma$
and $\m a$ is an algebra of signature $\sigma$,
then a \emph{$\phi$-irreducible} subset of $\m a$
is a subset $U\subseteq A$ such that whenever 
$\alpha(f)=n$ and $\phi(f)=i$ 
one has
\[
f^{\m a}(a_1,\ldots,a_n)\in U\Rightarrow a_i\in U.
\]

The complement of a $\phi$-irreducible subset
is called a \emph{$\phi$-ideal}. Explicitly, 
$I\subseteq A$ is a $\phi$-ideal if whenever
$\alpha(f)=n$,
$\phi(f)=i$ and $a_i\in I$, then 
$f^{\m a}(a_1,\ldots,a_n)\in I$.
\end{df}

In this terminology, 
a left ideal of a semigroup
with multiplication represented by the symbol $m$
would be a $\phi$-ideal for the function
$\phi\colon \{m\}\to \{1,2\}\colon m\mapsto 2$, while 
a right ideal would be a $\phi$-ideal for the function
$\phi\colon \{m\}\to \{1,2\}\colon m\mapsto 1$.

\begin{thm}\label{ideal}
Let $\m a$ be an algebra of signature $\sigma$
and let $\phi$ be a selector for $\sigma$.
If $\m a$ is the union of finitely many proper $\phi$-ideals,
then $d_{\m a}(n)\geq 2^n$.
\end{thm}

\begin{proof}
The union of $\phi$-ideals is again a $\phi$-ideal,
so if $\m a$ is the union of $k\geq 2$ proper $\phi$-ideals
then it can be expressed as the union $I\cup J$ of 2
proper $\phi$-ideals. The complements $I':=A\setminus I$
and $J':=A\setminus J$ are disjoint $\phi$-irreducible sets.
Any product $T:=X_1\times \cdots \times X_n$, with
$X_i=I'$ or $J'$ for all $i$,
is a $\phi$-irreducible subset of $A^n$.
Each such set must contain at least one element of any 
generating set, since the $\phi$-irreducibility
of $T$ implies that $A^n\setminus T$ is a subuniverse
of $\m a^n$. Since there are $2^n$ products
of the form $X_1\times \cdots \times X_n$ with
$X_i=I'$ or $J'$, and they are pairwise disjoint,
any generating set for $\m a^n$ must contain 
at least $2^n$ elements.
\end{proof}

\begin{exmp}
In this example, $\m 2$ is the 2-element Boolean algebra
and ${\m 2}^{\circ} = \lb \{0,1\}; \to\rb$ is the reduct
of $\m 2$ to the operation $x\to y = x'\vee y$.
The variety $\vr V$ generated by $\m 2^{\circ}$
is called the variety of implication algebras.
This variety is congruence
distributive and has $\m 2^{\circ}$ as its unique subdirectly
irreducible member. Each finite algebra in $\vr V$
may be viewed as an order filter in a finite Boolean algebra:
if $\m a\in {\vr V}_{\textrm{fin}}$, then an
irredundant subdirect representation $\m a\leq (\m 2^{\circ})^k$
may be viewed as a representation of $\m a$ as a subset
of $\m 2^k$ closed under $\to$; such subsets of $\m 2^k$
are order filters.

Considering an algebra $\m a\in {\vr V}_{\textrm{fin}}$
to be an order filter in $\m 2^k$, each 
order filter contained within $\m a$
is a left ideal in $\m a$ with respect to the operation $\to$.
By Theorem~\ref{ideal},
if $\m a$ is the union of its proper order filters,
its growth rate is exponential.
This case must occur unless $\m a$ itself
is a principal order filter in $\m 2^k$. Since we represented
$\m a$ irredundantly, $\m a$ is a principal order
filter in $\m 2^k$ only when it is the improper filter,
i.e., $\m a = (2^{\circ})^k$. In this situation 
$\m a$ is polynomially equivalent to the Boolean
algebra $\m 2^k$. It follows from Theorem~\ref{basic_estimates}~(1)
and the fact that $\m 2$ is primal that $\m 2^k$
has logarithmic growth rate.
In summary, a finite implication algebra has 
logarithmic growth rate if it has a least element
and has exponential growth rate otherwise.

Now, it is easy to produce a chain of implication
algebras $\m a_1\leq \m a_2\leq \cdots$, each one
a subalgebra of the next, where 
$\m a_i$ has logarithmic growth when $i$ is odd
and exponential growth when $i$ is even.
One simply chooses larger and larger 
Boolean order filters which are
principal only when $i$ is odd.
The following figure shows how the chain might begin.
\bigskip

\begin{center}
\vbox{
\begin{picture}(300,90)
\setlength{\unitlength}{1mm}

\put(0,20){\circle*{1}}
\put(0,30){\circle*{1}}

\put(0,20){\line(0,1){10}}

\put(20,20){\circle*{1}}
\put(30,20){\circle*{1}}
\put(30,30){\circle*{1}}

\put(20,20){\line(1,1){10}}
\put(30,20){\line(0,1){10}}

\put(50,10){\circle*{1}}
\put(50,20){\circle*{1}}
\put(60,20){\circle*{1}}
\put(60,30){\circle*{1}}

\put(50,10){\line(1,1){10}}
\put(50,10){\line(0,1){10}}
\put(50,20){\line(1,1){10}}
\put(60,20){\line(0,1){10}}

\put(80,10){\circle*{1}}
\put(80,20){\circle*{1}}
\put(90,20){\circle*{1}}
\put(90,30){\circle*{1}}
\put(100,20){\circle*{1}}

\put(80,10){\line(1,1){10}}
\put(80,10){\line(0,1){10}}
\put(80,20){\line(1,1){10}}
\put(90,20){\line(0,1){10}}
\put(100,20){\line(-1,1){10}}

\put(-2,2){$\m a_1$}
\put(23,2){$\m a_2$}
\put(53,2){$\m a_3$}
\put(88,2){$\m a_4$}

\put(11,2){$\leq$}
\put(38,2){$\leq$}
\put(71,2){$\leq$}
\put(103,2){$\cdots$}

\end{picture}
\medskip

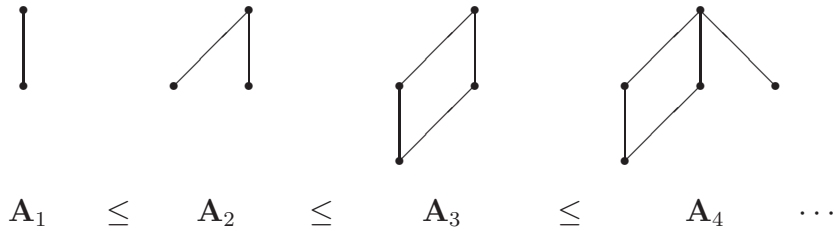
{{\small\sc Figure: 
A chain of implication algebras.}}
}
\end{center}
\end{exmp}

\section{Problems}\label{problems}

In this paper, we have filled in
one gap in knowledge about
the spectrum of possible growth rates of finite algebras
by producing examples 
with superlinear polynomial growth rates. There is an interesting gap
in knowledge that remains between logarithmic and linear growth rates.

\medskip
\noindent
%\begin{prb}
{\bf Problem 6.1.}
Is there a finite algebra $\m a$ where
$d_{\m a}(n)\notin \Omega(n)$ and
$d_{\m a}(n)\notin O(\log(n))$? 
%\end{prb}

\medskip

A special case that might be tractable is the 
following.

\medskip
\noindent
%\begin{prb}
{\bf Problem 6.2.}
Is there a 2-element partial algebra $\m a$ where
$d_{\m a}(n)\notin \Omega(n)$ and
$d_{\m a}(n)\notin O(\log(n))$? 
%\end{prb}

\medskip

We know that no finite algebra with a $0$-pointed cube term can
have growth rate between logarithmic and linear,
but do not know the situation for pointed 
cube terms. The following seems to be the most interesting
special case.

\medskip
\noindent
%\begin{prb}
{\bf Problem 6.3.}
Is it true that a finite algebra with a $2$-sided
unit for some binary term has logarithmic or linear growth?
%\end{prb}

\medskip

There is also a possible gap near the exponential end
of the spectrum.

\medskip
\noindent
%\begin{prb}
{\bf Problem 6.4.}
Is there a finite algebra $\m a$ where
$d_{\m a}(n)\notin 2^{\Omega(n)}$ and
$d_{\m a}(n)\notin O(n^k)$ for any $k$? 
%\end{prb}

\bibliographystyle{plain}

\end{document}